\documentclass{amsart}
\usepackage{amssymb, amsmath, amsthm}
\usepackage[all]{xy}
\usepackage{color}
\usepackage{graphics}
\usepackage[T1]{fontenc}
\usepackage[latin1]{inputenc} 
\usepackage{newlfont,verbatim, indentfirst,enumerate}
\usepackage{latexsym, amsfonts, amsbsy, mathrsfs, stmaryrd}
\usepackage{verbatim}
\usepackage{tikz} \usetikzlibrary{arrows}

\usepackage[colorinlistoftodos]{todonotes}

\newcommand{\boldrho}{{\pmb\rho}}
\newcommand{\boldalpha}{{\alpha}}
\newcommand{\boldgamma}{{\gamma}}
\newcommand{\boldpsi}{\pmb{\psi}}
\newcommand{\boldphi}{\pmb{\phi}}

\definecolor{laura}{rgb}{0.5, 0, 0.7}
\makeatletter
\def\oversortoftilde#1{\mathop{\vbox{\m@th\ialign{##\crcr\noalign{\kern3\p@}%
      \sortoftildefill\crcr\noalign{\kern3\p@\nointerlineskip}%
      $\hfil\displaystyle{#1}\hfil$\crcr}}}\limits}

\def\sortoftildefill{$\m@th \setbox\z@\hbox{$\braceld$}%
  \braceld\leaders\vrule \@height\ht\z@ \@depth\z@\hfill\braceru$}

\makeatother

\def \calC {\mathcal{C}}

\def \calG {\mathcal{G}}
\def \calH {\mathcal{H}}

\def \calO {\mathcal{O}}
\def \calR {\mathcal{R}}

\def \calU {\mathcal{U}}
\def \calV {\mathcal{V}}
\def \calW {\mathcal{W}}
\def \CU   {\calC(\mathfrak{U})} 
\def \GU   {\calG(\mathfrak{U})}
\def \RU   {\calR_{\mathfrak{U}}}

\def \tU   {\widetilde{U}}
\def \tV   {\widetilde{V}}
\def \tW   {\widetilde{W}}

\def \tlambda {\widetilde{\lambda}}
\def \tmu     {\widetilde{\mu}}
\def \trho    {\widetilde{\rho}}
\def \tkappa  {\widetilde{\kappa}}
\def \ttheta  {\widetilde{\theta}}
\def \tnu     {\widetilde{\nu}}
\def \tdelta  {\widetilde{\delta}}
\def \txi     {\widetilde{\xi}}
\def \tgamma  {\widetilde{\gamma}}

\def \Gbmap {\tau}
\def \Hbmap {\varepsilon}

\def \red   {\,\operatorname{red}}

\def \Aut   {\operatorname{Aut}}

\def \bbZ   {\mathbb Z}

\def \CONCR {\mathsf{Con}} \def \ABST {\mathsf{Abst}}

 \def \Groupprof {\mathsf{GroupMod}}

\newcommand {\unif} [1] {(\tU_{#1},G_{#1},\rho_{#1},\pi_{#1})}

\newcommand{\slashing}[2]{#1}   

\newcommand{\pullback}[2]{\mathbin{{}_{#1}\kern-\scriptspace{\times}_{#2}}}

\numberwithin{equation}{section}

\theoremstyle{plain}
\newtheorem{theorem}{Theorem}[section]

\newtheorem{thm}[theorem]{Theorem}
\newtheorem{prop}[theorem]{Proposition}
\newtheorem{lem}[theorem]{Lemma}   
\newtheorem{lma}[theorem]{Lemma}   
\newtheorem{cor}[theorem]{Corollary}

\theoremstyle{definition}
\newtheorem{defin}[theorem]{Definition}
\newtheorem{dfn}[theorem]{Definition}

\newtheorem{rmk}[theorem]{Remark}
\newtheorem{rmks}[theorem]{Remarks}

\newtheorem{exa}[theorem]{Example}
\newtheorem{Notation}[theorem]{Notation}

\begin{document}

\title[Atlases for Ineffective Orbifolds]
{Atlases for Ineffective Orbifolds}

\author{D.~Pronk, L.~Scull, M.~Tommasini}

\address{Department of Mathematics and Statistics,
Chase Building,
Dalhousie University, 
Halifax, NS, B3H 3J5, Canada}
\email{pronk@mathstat.dal.ca}

\address{Department of Mathematics, Fort Lewis College,  1000 Rim Drive, 
Durango, Colorado 81301-3999    
USA}
\email{scull\_l@fortlewis.edu}

\address{(last affiliation) Mathematics Research Unity, University of Luxembourg, Luxembourg}
\email{matteo.tommasini2@gmail.com}

\date{\today}
\subjclass[2010]{}

\thanks{The first author thanks NSERC for its financial support and Fort Lewis College for
its hospitality; the third author thanks FNR Luxembourg for the grant 4773242.}

\begin{abstract}
We give a definition of atlases for ineffective orbifolds, and prove that this definition
leads to the same notion of orbifold as that defined via topological groupoids.
\end{abstract}

\maketitle

\tableofcontents

\section{Introduction}

This paper considers general, not necessarily effective, orbifolds; we will refer to these as
ineffective.  Ineffective orbifolds have been studied from several perspectives, and multiple
definitions have been proposed.  If we define orbifolds using topological groupoids, it is
straightforward to generalize to ineffective orbifolds, and this is the approach that 
has often been used in the literature (for example,~\cite{ALR,CR1,CR2,FO}).  On the other hand,
it is nice to also have a traditional atlas definition, and here the generalization to the ineffective
case  becomes less obvious.  Although it is easy to define an ineffective orbifold chart, figuring
out what the chart embeddings are and how to build an atlas out of them is less obvious.  As it
stands, the ineffective atlas definitions currently used in the literature do not match the objects
described using topological groupoids; this has been observed in~\cite{HM}.  

In this paper, we give a new definition of orbifold atlas
for ineffective orbifolds, generalizing the definition of orbifold atlas in the effective case,
and show that it does match  the topological groupoid approach.  Our definition will not change
the existing notion of chart for ineffective orbifolds, but will change the structure of
the embeddings between charts that create the atlas.  

The definition we develop is considerably longer and less straighforward than the one for ineffective orbifold groupoids.  We believe that  it is still interesting to have and that it sheds new light on the structure of ineffective orbifolds.    In addition, although writing out the definition in full generality is quite long, applying it in particular examples becomes more tractable, as we will show in Example \ref{E:z3p2}.

We see two avenues for immediate applications.  The first is in considering notions of suborbifolds and embeddings.     In the effective case, this has been studied from both the groupoid and atlas perspective.  The original definitions  were given via local charts,  and it turns into a surprisingly delicate question to get the definitions correct even on this level.  The original definition excludes natural examples that one would want to be included, and therefore has been refined and expanded;   see  for example \cite{BB}, who use the atlas and chart approach to develop examples and discuss the subtleties and various notions used here.   These ideas have also  been studied from a groupoid approach.  Once again, the natural notion of suborbifold is very restrictive and excludes obvious examples, such as the diagonal of a product orbifold.   Thus, again the idea has been expanded.  For example,   \cite{CHS} develops a  groupoid approach to orbifold embeddings that encompass some of the desired examples, but only manages to fully develop the theory for translation groupoids.   Given this, we believe that having a fully developed notion of atlases for ineffective orbifolds is a  necessary first step at considering these questions in the ineffective setting. 

A second avenue of investigation comes from the idea of developing an orbifold category.  This idea is inspired by the work of 
Grandis~\cite{G} (and further developed by Cockett et al.~\cite{CCG}), to generalize the idea of manifolds.  Based on local smooth structures pieced together via join restriction categories, it is possible to define manifold objects in a variety of categories.      We are interested in extending this to the idea of generalized orbifold categories, based on locally defined notions of chart and atlas.  

\subsection{Organization} The paper is structured as follows.  We start by giving background on the various currently existing definitions
of orbifolds, and explaining the issues surrounding the existing definitions in Section~\ref{sec-2}.
In Section~\ref{sec-3}, we give some background on the language of group bimodules, which we
use in defining our atlases.   We give the new orbifold atlas definition in Section~\ref{S:defn}.
The remainder of the paper is devoted to proving that this definition does match the orbifolds
defined via topological groupoids:  in Section~\ref{sec-5} we show that we can construct an
orbifold groupoid (that is, one which is  \'etale proper and Lie)
from one of our atlases, and in Section~\ref{S:atlastogroupoid} we show that we can construct one of our
orbifold atlases from a groupoid which defines an orbifold.  In Section~\ref{sec-7} we prove that
equivalent orbifold atlases correspond to Morita equivalent groupoids, completing the
correspondence between our atlas definition and the groupoid approach.  

\subsection{Acknowledgements}  We would like to thank the referee for very helpful suggestions that allowed us to considerably streamline and clarify the arguments in Sections 6 and 7.

\section{Motivation:  Orbifolds via Atlases and via Groupoids}\label{sec-2}
\subsection{Effective Orbifolds}
Orbifolds are spaces which are locally modeled by quotients of finite groups acting smoothly on
open subsets of $\mathbb{R}^n$.  Conceived as generalizations of manifolds, they were originally
described in the same way that manifolds are, using  the language of charts and atlases. 
 
In Satake's original definition~\cite{Sa1}, an orbifold chart (or uniformizing system) 
for an open subset $U$ of a topological space $X$ consists of  $\tU$, a connected open subset of
$\mathbb{R}^n$, with a finite group $G$ acting \emph{effectively} on $\tU$, such that the
quotient space $\tU/G$ is homeomorphic to $U$ via the projection $\pi$.   
An embedding of charts $(\tU_1,G_1,\pi_1)\hookrightarrow(\tU_2,G_2,\pi_2)$ is defined by a
smooth embedding $\lambda:  \tU_1 \to \tU_2$ such that $\pi_1 = \pi_2 \circ \lambda$, so that it
induces a continuous embedding on the quotient spaces.   An atlas for a space then consists of a
collection of charts such that the quotients cover the underlying space, with all chart embeddings
between them, and suitable compatibility conditions (see Section \ref{S:efforb} for full definition).  We will refer to such a structure as a Satake atlas.

These atlases have several nice properties that have proved very useful when working with effective
orbifolds.  First, given an open subset $\tU$ of $\mathbb{R}^n$ and an effective action of a finite
group $G_U$ on $\tU$, we can think of this as an orbifold chart for $U = \tU/G_U$.   Then the
collection of chart embeddings from $\tU$ into itself has the structure of the group $G_U$: each
embedding $\lambda$  is of the form $x\mapsto g\cdot x$ for a unique $g\in G_U$, and each group
element gives a chart embedding, such that the composition of chart embeddings corresponds to
the multiplication in the group $G_U$.  Furthermore, since for each chart embedding $\lambda$ and
each group element $g\in G_U$ the map $\lambda g$ is a chart embedding as well, each embedding
gives rise to a unique group homomorphism $\phi: G_U \to G_U$ given by conjugation by the element
$g$, such that  the embedding $\lambda$ is equivariant with respect to $\phi$ in the sense that
$\lambda(g\cdot x) = \phi(g)\cdot\lambda (x)$.  These results about the interaction of the
embeddings with the structure groups were first given in~\cite{Sa2} for actions with fixed sets
of codimension at least 2 and proved in full generality in~\cite[Section 1 and the appendix]{MP}.

More generally, these references show that if we have any chart embedding $\lambda:\tU_1 \to
\tU_2$, we can realize this as an equivariant map:  there is a unique induced homomorphism $\phi:
G_1 \to G_2$ such that $\lambda(g\cdot x) = \phi(g)\cdot\lambda (x)$ for all $x\in \tU_1$.

Given any open connected subset $V$ of a chart $U$ with structure group $G_U$, it is always possible
to endow $V$ with an induced orbifold chart structure.  To do this, we take  a connected component
$\tV$ of the inverse image of $V$ in the covering space $\tU$, and consider the action of the
subgroup $G_V$ of $G_U$  that leaves this subset  $\tV$ invariant (but not necessarily pointwise
fixed).   Using this fact, a slight variation on the definition of orbifold atlas was introduced
by~\cite{FO} where an atlas is given as a collection of orbifold charts which covers the space
$X$ and is \emph{locally compatible} in the sense that all charts in the collection containing a
given point $x$ induce the same germ of a chart at that point.  

Both these definitions of atlases are rather cumbersome to work with, particularly when it comes
to defining maps between orbifolds with desirable properties;  see~\cite{Po} for instance.  
Therefore, an alternate description of the orbifold category has been developed
using topological and differentiable groupoids.  A topological groupoid has a space of objects and
a space of arrows identifying certain objects;  the quotient of the object space by these
identifications gives the underlying quotient space of the orbifold, and the identifications
are used to encode the local singularity structure.  To give a smooth structure we require
our orbifold groupoids to be Lie groupoids (with smooth structures and maps).  To keep the
singularities of the type desired for orbifolds, that is, those given by quotients of smooth
finite group actions, we place restrictions on the groupoid and ask that the diagonal 
be a proper map and that the groupoid be \'etale (i.e., such that the source and target maps 
from the space of arrows to the space of objects are local diffeomorphisms).    When looking at
effective orbifolds, we also require the groupoids to be effective.  Two such groupoids represent
equivalent orbifolds precisely when they are Morita equivalent~\cite{MP}.
  
The groupoid approach and the original atlas idea lead to the same notion of {\em effective} orbifold.  
If we start with an effective orbifold atlas, we can create a groupoid  by defining the object space
to be a disjoint union of the chart spaces $\tU$, and the arrow space to consist of equivalence
classes of spans of chart embeddings $\tU_1 \leftarrow \tV \to \tU_2$ \cite{Pr1,PSi}.  
In particular, if we start with an orbifold atlas that consists of a  single chart given by the
quotient of an effective action of a finite group $G_U$ on an open subset $\tU$ of $\mathbb{R}^n$,
then as discussed above, the chart embeddings from $\tU$ to itself form a group which is isomorphic
to the structure group $G_U$.  Hence the groupoid obtained is the translation groupoid
$G_U \ltimes \tU$, encoding the quotient orbifold as expected.  
It was shown in~\cite{MP} that this leads to an equivalence between the category of effective
\'etale Lie groupoids, localized with respect to Morita equivalences, and the category of orbifolds
created from effective atlases (endowed with a suitable notion of maps between orbifolds).  
  
\subsection{Ineffective Orbifolds}
More recently, people have become interested in ineffective orbifolds, and examples of these have
become important in various applications arising from mathematical physics.  
From a groupoid perspective it is relatively easy to define an ineffective orbifold.  Just as the
groupoids representing effective orbifolds are effective smooth \'etale groupoids with a proper
diagonal,   so the appropriate notion of \emph{ineffective} orbifold would be to take a quotient
of an \emph{ineffective} smooth \'etale groupoid with proper diagonal (and again call two of
them equivalent if there is a Morita equivalence between the groupoids).
This definition of ineffective orbifold has been used frequently (see~\cite{ALR,CR1,CR2,FO}
for instance)  and  leads immediately to a good  notion of a map between ineffective orbifolds.  

We can define a chart by simply dropping the effective condition
from the definition given above. When we do so, however,  an equivariant embedding no longer
dictates a unique group homomorphism, so we need to consider what chart embeddings to include
in our atlas.  One generalization that has been used is to include the group homomorphism as
part of the embedding information, and to define an embedding of orbifold charts as a group
homomorphism $\phi:  G_1 \to G_2$ together with a map   $\lambda:  \tU_1 \to \tU_2$ which is
equivariant in the sense that $\lambda(g\cdot x) = \phi(g)\cdot\lambda(x)$, and furthermore
require that $\phi$ induces an isomorphism on the kernels of the actions~\cite{CR1}.
  
Although we want each chart embedding to be equivariant in this sense, simply taking this as a
definition of embeddings does not lead us immediately to a satisfactory definition of an
orbifold atlas as a whole.   In particular, we do not necessarily obtain the result (discussed
above in the effective case) that the group of embeddings from a chart to itself is isomorphic to the
structure group of the chart.  For example, consider an open  subset $\tU$ of $\mathbb{R}^n$ with a
trivial action of a group $G$, so that $\tU = U$ and the orbifold is completely ineffective.   
In this case, the embedding of $\tU$ to itself must be the identity on the space, and any
automorphism $\phi$ of the group $G$ can be used.   So if $G$ is $\mathbb{Z}/2$, there are no
non-trivial automorphisms of $G$ and so there is only a single orbifold embedding from $\tU$ to
itself.   On the other hand, if $G$ is $\mathbb{Z}/2 \oplus \mathbb{Z}/2$, the group of
automorphisms is isomorphic to $S_3$ rather than $\mathbb{Z}/2 \oplus \mathbb{Z}/2$, and so we
obtain more chart embeddings
than the size of the structure group.   In both of these cases, if we construct a groupoid based
on the atlas of this single chart and its self-embeddings, we will \emph{not} get the
translation groupoid  $G \ltimes\tU$.  Thus, the correspondence between the atlas viewpoint and
the groupoid viewpoint has broken down.  This is the discrepancy observed also by Henriques
and Metzler in~\cite{HM}.
     
We can also consider whether the alternate orbifold atlas point of view, based on local
compatibility, may have a better correspondence with the groupoids;  this approach 
was used in~\cite{CR1} for instance.   Unfortunately, there are smooth \'etale groupoids with
proper diagonal which are not Morita equivalent, but  whose quotient spaces are homeomorphic and
whose associated atlas charts give the same germ structure at corresponding points.
So these orbifolds would be the same when considered as spaces with germs of atlas charts as in~\cite{CR1}, but
they are not the same as `groupoid orbifolds'.    The following example illustrates this.  

\begin{exa}\label{E:z3}
We define groupoids $\calG$ and $\calH$ representing completely
ineffective orbifolds with  $\mathbb{Z}/3$ isotropy at each point on a circle $S^1$.   Both
groupoids have object space given by the circle $S^1$.  For the first groupoid,  $\calG$,  let
the space of arrows be $\mathbb{Z}/3\times S^1$ (where $\mathbb{Z}/3$ has the discrete topology) 
where the source and target maps are the projection, $s(a,x)=x=t(a,x)$.  We define the groupoid
composition $m((a,x),(b,x))=(ab,x)$, where $ab$ is taken as the product in $\mathbb{Z}/3$.
For $\calH$ we define the space of arrows to be the disjoint union of two copies of $S^1$.
One circle $X$ will represent the identity arrows, with source and target maps defined by the
identity $S^1\to S^1$.  The second circle $Y$ will represent the non-identity arrows (which still
act trivially, since this is a completely ineffective orbifold).   Both the source and target maps 
will be the double wrapping map $S^1\to S^1$  that make $Y$ into a two-fold cover
of the object space $S^1$.   Composition  $m(a,b)$ for arrows $a$ and $b$ is only defined when
the pair of points lies in the same fiber over the space of objects (since $s=t$).  
So let $a,b$ be two points in the same fiber.   If $a\in X$, $m(a,b)=b$ since $X$ is the identity
component, and similarly, if $b\in X$ then $m(a,b)=a$.   If $a=b\in Y$, then $m(a,b)=c$ where
$c$ is the other point in $Y$ in the same fiber as $a=b$.   
If $a,b\in Y$ and $a\neq b$, then $m(a,b)=c$ where $c$ is the unique point in the identity
component $X$ in the same fiber as $a$ and $b$.   The reader may verify that for both groupoids
the germs of the charts will give atlas charts  on $S^1$ with structure group $\mathbb{Z}/3$, acting
ineffectively.   However, $\calG$ and $\calH$ are not Morita equivalent, and so do not define the
same orbifold using the groupoid definition.  
\end{exa}

\subsection{Goal of the Paper}   
The goal of this paper is to give an atlas definition that does match up with the groupoid
(and hence, stack) perspective of orbifolds, and agrees in a suitable sense with the classical
Satake definition in the effective case.      We have shown in the previous examples  that
the existing definitions of ineffective chart embeddings do not lead to the same objects as the
groupoids, in either the Satake or local definitions of atlases.  Therefore we will reconsider
what a chart embedding should be in a definition of an orbifold atlas.   We will use the
language of  bimodules to do this, as discussed in the next section.

\section{Background: Group Bimodules}\label{sec-3}
When we create a groupoid presentation of an orbifold starting from an atlas, the structure of the groupoid
depends on the collection of all the embeddings from one chart into another (or to itself).
This collection will have the structure of a group bimodule, which we develop in this section.  

Given any $G$-set $X$ and $H$-set $Y$ and any group homomorphism $\phi:G \to H$,
we can consider the set of  $\phi$-equivariant maps, i.e.\ maps $f:  X \to Y$ that satisfy
$f(g\cdot x)=\phi(g)\cdot x$.   Each such map generates further equivariant maps $fg$ and $hf$
for any $g\in G$ and $h\in H$, giving a  set of maps $\mathsf{f}$ with commuting left $H$- and
right $G$-actions.  These equivariant maps form a bimodule with the left- and right-actions
tying them together. (Note that we use the word `bimodule' for sets with compatible left and
right group actions; we do not assume that there is a notion of addition to make them
abelian groups.)

Recall that in the previous attempt at defining an ineffectively atlas,  a chart embedding
$(f,\phi)\colon (\tU,G,\pi_U)\to(\tV,H,\pi_V)$ consisted of
an injective  group homomorphism $\phi\colon G \to H$ with a $\phi$-equivariant map $f \colon\tU\to\tV$ between
the spaces.  Our approach will be to consider an equivariant chart embedding $(f,\phi)$
as part of a larger structure, the bimodule $\mathsf{f}$, and to consider this entire set of maps
with its left and right actions as part of the structure.  Our  atlas definition will be 
based on these bimodule chart embeddings.  

\subsection{Group Bimodules}
For finite groups
$G$ and  $H$,  a  bimodule $\boldphi\colon G\longarrownot\longrightarrow H$ is given by a (non-empty) set $\mathsf{M}$,
together with a right $G$-action and a left
$H$-action that are compatible, i.e., such that for $m\in \mathsf{M}$, $g\in G$ and $h\in H$, we
have $(h\cdot m)\cdot g=h\cdot(m\cdot g)$.  Thus, it is  a set
with a left $H$-, right $G$- action, which we call a left $H$-, right $G$- module, 
or simply a bimodule if the actions are clear from the context.  

We denote the bicategory of such bimodules by  $\Groupprof$.   In this bicategory, the identity 
$G \longarrownot\longrightarrow G$ is defined by the trivial bimodule $\mathsf{G}$, given by the set
$\mathsf{G} = G$ together with left and right actions defined by multiplication.    

Composition of  $\Groupprof$ is defined by a tensor product:  given 
$\boldphi:{G}\longarrownot\longrightarrow{H}$ and $\boldpsi:{H}\longarrownot\longrightarrow{K}$
defined by the bimodules $\mathsf{M} $ and
$\mathsf{N} $ respectively, their composition $\boldpsi
\circ\boldphi$ is given by the tensor product bimodule $\mathsf{Q} = \mathsf{N} \otimes_H
\mathsf{M} =  (\mathsf{N} \times \mathsf{M}) /\sim $, 
where $\sim$ is the equivalence relation  $(y\cdot h,x)\sim(y,h\cdot x)$
for each $x\in \mathsf{M}, y\in \mathsf{N}$ and $h\in H$.  The  actions are defined by 
$k\cdot\left[y,x\right]\cdot g=\left[k\cdot y,x\cdot g\right]$ for $g\in G$
and $k\in K$.  We will write $y\otimes x$ for any equivalence class $[y,x]$ in $\mathsf{Q}$.

A 2-cell or natural transformation in the bicategory is defined by a map
of bimodules $\boldalpha: \mathsf{M} \to\mathsf{N}$, i.e., a set map
satisfying $\boldalpha(h\cdot m\cdot g)=h\cdot\alpha(m)\cdot g$.  Under these definitions, $\Groupprof$ is a bicategory.

  \subsection{Group Homomorphisms and Bimodules}\label{SS:gphom}
Let $\phi: G \to H$
be a group homomorphism.     We consider the induced bimodule 
$\boldphi:  G \longarrownot\longrightarrow H$.  This bimodule 
is defined to be the set $H$, together with the free and transitive left action of $H$,
and the right action of $G$ defined by $h \cdot g = h \phi(g)$.   Under this correspondence,
the right action of $G$ is free, respectively transitive, if and only if the original group
homomorphism $\phi$ is injective, respectively surjective.   We can think of $\boldphi$ as the set consisting of maps  $G\to H$ that are of the form $g\mapsto h\phi(g)$.

Every bimodule $\boldpsi\colon {G}\longarrownot\longrightarrow{H}$ for which $H$ acts freely and transitively 
is of this form, although not in a canonical way.   The set $\boldpsi$ is an
$H$-torsor;  once we choose its `base point', i.e., identity element $e$, the group
homomorphism $\psi:G\rightarrow H$  inducing $\boldpsi$ can be found in the following way:
$\psi(g)$ is the unique element $h\in H$ such that $e\cdot g=h\cdot e$.
Note that a different choice of base point will produce a conjugate homomorphism.
So we can think of these particular bimodules as corresponding to conjugacy classes of
homomorphisms $G\rightarrow H$.

\subsection{Atlas Bimodules}\label{SS:athom}
 We now  return to our initial consideration of  how an atlas may be constructed from chart
embeddings, where the full collection of embeddings from one chart to another comes with actions
of the structure groups of the source and target charts.      In the effective case, any two
chart embeddings are always related by a unique  element of the codomain structure group, so that
one embedding is obtained from the other by the action of this group
element~\cite[Proposition~A.1]{MP}.   That is, the action  of the codomain group is always
both free and transitive on the set of embeddings.  When we move to the more general
ineffective setting, we will want to ensure that our chart embedding bimodules continue to
have both of these important properties.   Freeness of the action of the codomain
group ensures that the chart embeddings can encode the orbit structure in the codomain.
Transitivity ensures that all embeddings that make up the bimodule give the same maps
on the quotient space.      
 
Lastly, we consider the action of the domain group.  Since the codomain action is free and
transitive, we have seen in Section~\ref{SS:gphom} that such a bimodule corresponds to a 
conjugacy class of group homomorphisms.   For an embedding, we want these group homomorphisms
to be injective, and so we require that the action of the domain structure group is free. 
 
Therefore we make the following definition. 
 
\begin{defin}
Given two finite groups $G$ and $H$, a bimodule $\mathsf{M}: G \longarrownot\longrightarrow H$  is an
\emph{atlas bimodule} if it has the following properties:

\begin{itemize}
 \item it is non-empty;
 \item the left $H$-action is free and transitive;
 \item the right $G$-action is free.
 \end{itemize} 
\end{defin}  
 
It follows from the discussion in Section~\ref{SS:gphom} that we have the following result.
 
\begin{lem}\label{lem-01}
For each atlas bimodule $\mathsf{M}:  G \longarrownot\longrightarrow {H}$  and for each $m\in \mathsf{M}$, there
is an induced injective group  homomorphism $\Lambda_m:G\rightarrow H$, such that for each
$g\in G$ we have $m\cdot g=\Lambda_m(g)\cdot m$.
\end{lem}

Moreover, any two such group homomorphisms defined by $\mathsf{M}$ are related by conjugation:  
if $m' \in \mathsf{M}$ is such that $m' = hm$, then $\Lambda_{m'}(g) = h\Lambda_m(g)h^{-1}$.
So the atlas bimodule defines a conjugacy class of  injective group homomorphisms $G \to H$.   

\section{The New Atlas Definition}\label{S:defn}
\subsection{Satake Atlases for Effective Orbifolds} \label{S:efforb}
We begin by giving the formal definition of an atlas in the effective case to use as
comparison.

In Satake's original definition~\cite{Sa2}, an \emph{effective orbifold chart}
(or \emph{effective uniformizing system}) for an open subset $U$ of a topological space $X$
consists of a triple $(\tU, G,\pi)$, where:

\begin{itemize}
 \item $\tU$ is a connected open subset of $\mathbb{R}^n$ 
 \item $G$ is a finite group acting \emph{effectively} on $\tU$ \item $\pi:\tU\rightarrow U$ is a 
   continuous and surjective map that induces a homeomorphism between $U$ and $\tU/G$.  
\end{itemize}

An embedding of charts $(\tU_1,G_1,\pi_1)\hookrightarrow(\tU_2,G_2,\pi_2)$ is defined by a 
smooth embedding $\lambda:  \tU_1 \to \tU_2$ such that $\pi_1 = \pi_2 \circ \lambda$.
An atlas for a space consists of a collection of charts $\calU$ such that the quotients cover
the underlying space, and the collection of all chart embeddings between them.  
Furthermore, the charts are required to be locally compatible in the sense that for any two charts 
for subsets $U,V\subseteq X$ and any point $x\in U\cap V$, there is a neighbourhood $W\subseteq
U\cap V$ containing $x$ with a chart $(\tW,G_W,\pi_W)$ in $\calU$, and chart embeddings into
$(\tU, G_U,\pi_U)$ and $(\tV, G_V,\pi_V)$.      We will refer to such atlases, charts and chart
embeddings as Satake atlases, Satake charts and Satake embeddings respectively.

It was shown in~\cite{Sa2} and~\cite{MP} that whenever $(\tU, G_U,\pi_U)$ and $(\tV, G_V,\pi_V)$ 
are effective charts with $U\subseteq V$ and $\tU$ is connected and simply connected, there is at least one
chart embedding $\lambda\colon \tU\to\tV$.   We will generally assume that our atlas charts are
connected and simply connected so that we have this property.

There are a couple of additional properties of these atlases that we will use below and we will
list here to make this paper more self-contained.
The first result was mentioned in~\cite{Sa2} and proved in full generality in~\cite{MP}.

\begin{lma}\label{overlap}
If $\lambda$ and $\lambda'$ are chart embeddings $\tU\to\tV$ and their images overlap, i.e.,
$\lambda(\tU)\cap\lambda'(\tU)\neq\varnothing$, then there is an element $g\in G_U$ such that
$\lambda=\lambda'\circ g$.
\end{lma}

This result can be used to prove the following.

\begin{lma}\label{factorization}
Given charts $U_i\subseteq U_j\subseteq U_k$ with $\tU_i$ simply connected and connected, and
embeddings $\lambda_{ki}\colon\tU_i\to\tU_k$ and $\lambda_{kj}\colon \tU_j\to\tU_k$ such that
$\lambda_{ki}(\tU_i)\cap\lambda_{kj}(\tU_j)\neq\varnothing$, there is an embedding
$\lambda_{ji}\colon\tU_i\to\tU_j$ such that $\lambda_{kj}\circ\lambda_{ji}=\lambda_{ki}$. 
\end{lma}

\begin{proof}
Let $x_k\in\lambda_{ki}(\tU_i)\cap\lambda_{kj}(\tU_j)\subseteq\tU_k$.   Furthermore, let
$x_i\in\tU_i$ and $x_j\in \tU_j$ be the unique points such that $\lambda_{ki}(x_i)=x_k$ and
$\lambda_{kj}(x_j)=x_k$.  Now let $\nu\colon\tU_i\to\tU_j$ be an embedding such that $\nu(x_i)=x_j$.
Then $\lambda_{kj}\circ\nu(\tU_i)\cap\lambda_{ki}(\tU_i)\neq\varnothing$.   By Lemma~\ref{overlap}
there is an element $g\in G_i$ such that $\lambda_{kj}\circ\nu\circ g=\lambda_{ki}$.
So $\lambda_{ji}:=\nu\circ g$ has the required property.
\end{proof}

The third result can be found in~\cite{Pr1} and~\cite{PSi} and is the key result in the
construction of the effective groupoid associated to a Satake atlas.   It is a stronger version of
local compatibility property of charts in atlases, called {\em strong local compatibility},
which follows from the atlas properties:

\begin{lma}\label{strong_comp}
Given charts $\tU_1,\tU_2,\tU_3$ in an atlas $\calU$, with chart embeddings
$\lambda_{31}\colon$ $\tU_1\to\tU_3$ and $\lambda_{32}\colon\tU_2\to\tU_3$ and points
$x_1\in \tU_1$ and $x_2\in \tU_2$ with $\lambda_{31}(x_1)=\lambda_{32}(x_2)$, there is
a chart $\tU_4$ in $\calU$ with embeddings  $\lambda_{i4}\colon\tU_4\to\tU_i$ for $i=1,2$,
such that $\lambda_{31}\circ\lambda_{14}=\lambda_{32}\circ\lambda_{24}$.
Moreover, there is a point $y$ in $\tU_4$, such that $\lambda_{i4}(y)=x_i$ for $i=1,2$.
\end{lma}

\subsection{Ineffective Atlas Definition} 
We give our new definition for an orbifold atlas in the language of atlas
bimodules of the previous section.   We will use the notion of chart for  ineffective
orbifolds  that has become standard in the literature. 

\begin{defin}\cite{ALR, CR1, CR2, LU}
Let $U$ be a non-empty connected topological space; an \emph{orbifold chart} (also known as a
\emph{uniformizing system}) of dimension $n$ for $U$ is a quadruple $(\tU,G,\rho,\pi)$ where:

\begin{itemize}
 \item $\tU$ is a connected and simply connected open subset of $\mathbb{R}^n$;
 \item $G$ is a finite group;
 \item $\rho:G\rightarrow\Aut(\tU)$ is a (not necessarily faithful) representation of $G$
  as a group of smooth automorphisms of $\tU$; we set $G^{\red}:=\rho(G)\subseteq\Aut(\tU)$ and 
  $\operatorname{Ker}(G):=\operatorname{Ker}(\rho)\subseteq G$;
 \item $\pi:\tU\rightarrow U$ is a continuous and surjective map that induces a homeomorphism
  between $U$ and $\tU/G^{\red}$.
\end{itemize}
\end{defin}

Associated to every orbifold chart $(\tU,G,\rho,\pi)$ we have the Satake chart $(\tU,G^{\red},\pi)$ 
with the reduced group $G^{\red}$; we will refer to this also as the {\em reduced chart}.   
An orbifold atlas $\mathfrak{U}$  for a space $X$ will contain  a collection $\calU$  of such
orbifold charts such that the underlying quotient spaces are open in $X$ and cover all of $X$,
and in fact, we will require that the associated reduced charts form a Satake atlas.

We write $\calO(X)$ for the lattice of all open subsets of $X$ and given an atlas $\calU$ we will  define  $\calO(\calU)$ to be the category with objects the charts of  $\calU$, with  morphisms given by inclusions of the underlying quotient sets.    The poset $\calO(\calU)$ is not required to be a lattice, as it is obvious that
it is not closed under unions, and contrary to the case of manifolds it may not be closed
under intersections either.

The usual local compatibility condition for Satake atlases requires that for any two charts $U$
and $V$ and point $x\in U\cap V$, there is a chart $W\subseteq U\cap V$ containing $x$.
This can be phrased in terms of the elements of the posets $\calO(X)$ and $\calO(\calU)$ as follows:
$$ \bigcup \{W\in\calO(\calU);\,W\subseteq U,W\subseteq V\}=U\cap V\quad\mbox{ in }\calO(X). $$
We will still require this for ineffective orbifolds.
We will also consider these posets as categories with at most one arrow between any two objects: we
write $\mu_{VU}\colon U\to V$ whenever $U\subseteq V$. 
When we are working with an indexed family of charts $U_i$ with $i\in I$, 
we will write $\mu_{ji}$ for $\mu_{U_jU_i}$.

For a Satake atlas with connected, simply connected charts, whenever there is an arrow $\mu_{ji}
\colon U_i\to U_j$ between the quotient open subsets of two orbifold charts, there is a Satake
embedding $\lambda_{ji}\colon \tU_i\to\tU_j$ (i.e., with $\pi_i=\pi_j\circ\lambda_{ji}$).
We will also call these embeddings {\em concrete}.   We will show that the set of all
these embeddings form an atlas bimodule $G_i^{\red}\longarrownot\longrightarrow G_j^{\red}$. 

\begin{defin}
Fix orbifold charts $\unif{1}$ and $\unif{2}$ for open sets $U_1$ and $U_2$ of a given
topological space $X$, with $U_1\subseteq U_2$.   Then we define the set of {\em concrete
embeddings} \[\mbox{\it Con}(\mu_{21}):=\left\{\textrm{all smooth embeddings }\lambda:\tU_1
\longrightarrow \tU_2\textrm{ s.t. }\pi_1=\pi_2\circ\lambda\right\}.\]
\end{defin}
This set coincides with the set of all Satake embeddings from the reduced
chart $(\tU_1,G_1^{\red},\pi_1)$ to the reduced chart $(\tU_2,G_2^{\red},\pi_2)$.


\begin{lem}\label{D:4.6}
Let $X$ be a topological space with a collection of orbifold charts
$\calU=\{(\tU_i,G_i,\rho_i,\pi_i)\}_{i\in I}$ of dimension $n$, such that  the reduced charts 
$\{(\tU_i,G_i^{\red},\pi_i)\}_{i\in I}$ form a Satake atlas $\calU^{\red}$.
For any $\mu_{ji}$ in $\calO(\calU)$, the set $\mbox{\it Con}(\mu_{ji})$ forms an atlas
bimodule $G_i^{\red}\longarrownot\longrightarrow G_j^{\red}$ with actions given by
composition. We will denote this bimodule by
\[{\CONCR}(\mu_{ji}):G_i^{\red}\longarrownot\longrightarrow G_j^{\red}.\] 
Furthermore, if $i=j$ the atlas bimodule $\CONCR(\mu_{ii})$ is isomorphic to the trivial
bimodule $\mathsf{G}_i^{\red}$ associated to the group $G_i^{\red}$.
\end{lem}

\begin{proof}
As noted by Satake~\cite{Sa2} and proved in full generality in~\cite[Proposition~A.1]{MP},
for any two atlas embeddings between effective charts 
$\lambda,\lambda'\colon \tU_i\rightrightarrows\tU_j$, there is
a unique group element $\rho_j(g)\in G_j^{\red}$ such that $\rho_j(g)\circ\lambda=\lambda'$.
So the action by $G_j^{\red}$ is free and transitive.   The fact that the action by $G_i^{\red}$
on this set is free follows from the fact that its action on $\tU_i$ is effective.
\end{proof}

For Satake orbifolds, and hence for the reduced charts of ineffective orbifolds, this
family of bimodules  is compatible in the following sense:\\

\begin{lem}\label{lem-06}
Let $X$ be a topological space with a collection of orbifold charts
$\calU=\{(\tU_i,G_i,\rho,\pi_i)\}_{i\in I}$ of dimension $n$, such that the  reduced charts 
$\{(\tU_i,G_i^{\red},\pi_i)\}_{i\in I}$  form a Satake atlas $\calU^{\red}$.  
Then there is a pseudofunctor
\[\CONCR:\calO(\calU)\longrightarrow\Groupprof,\]
defined on objects by $\CONCR(U_i):=G_i^{\red}$ and on morphisms by the atlas bimodules
\[\CONCR(\mu_{ji}):G_i^{\red}\longarrownot\longrightarrow G_j^{\red}\]
described above.  
\end{lem}

\begin{proof}
In order to make $\CONCR$ a pseudofunctor, we need to equip it with a coherent family of
composition and unit 2-cells.   Specifically, we need  natural isomorphisms of bimodules
$\boldgamma_{kji}\colon\CONCR(\mu_{kj})\circ\CONCR(\mu_{ji}) \Rightarrow\CONCR(\mu_{ki})$
for each composable pair of arrows $U_i\stackrel{\mu_{ji}}{\to}U_j\stackrel{\mu_{kj}}{\to}U_k$
in $\calO(\calU)$, and natural isomorphisms $\gamma_i$ from $\mathsf{G}_i^{\red}$ to
$\CONCR(\mu_{ii})$, where $\mathsf{G}_i^{\red}$ is the trivial bimodule which represents
the identity arrow in $\Groupprof$ for $G_i^{\red}$.

Recall that $\CONCR(\mu_{kj})\circ\CONCR(\mu_{ji})$ is given by the atlas bimodule 
$\CONCR(\mu_{kj})\otimes_{G_j^{\red}}\CONCR(\mu_{ji})$.   We define $\gamma_{kji}$ first as
a set map from $\CONCR(\mu_{kj})\otimes_{G_j^{\red}}\CONCR(\mu_{ji})$ to $\CONCR(\mu_{ki})$
by taking the equivalence class $\lambda_{kj}\otimes\lambda_{ji}$ to the composition
$\lambda_{kj}\circ\lambda_{ji}\in\CONCR(\mu_{ki})$ (where $\lambda_{kj}\in \CONCR(\mu_{kj})$
and $\lambda_{ji}\in\CONCR(\mu_{ji})$).  
It is easy to see that this is a well defined map; moreover, using~\cite[Proposition~A.1]{MP}
and Lemma~\ref{overlap} we get that $\gamma_{kji}$ is a bijection compatible with the actions
of $G_i^{\red}$ and $G_k^{\red}$, i.e., it corresponds to an invertible $2$-cell in
$\Groupprof$ as desired.   

The $\gamma_i\colon\mathsf{G}_i^{\red}\Rightarrow\CONCR(\mu_{ii})$ are given 
by the isomorphisms mentioned in Lemma~\ref{D:4.6}.

A straightforward computation shows that the data of the $\gamma_{kji}$ for each
composable pair of arrows $U_i\stackrel{\mu_{ji}}{\to}U_j\stackrel{\mu_{kj}}{\to}U_k$ in
$\calO(\calU)$, the $\gamma_i$ for each object $U_i$ in $\calO(\calU)$, together with the data
of $\CONCR(U_i)$ and $\CONCR(\mu_{ji})$ form a pseudofunctor $\CONCR$ from $\calO(\calU)$
to $\Groupprof$.
\end{proof}

Note that $\CONCR$ cannot be defined as a strict  functor, since the sets $\CONCR(\mu_{kj})
\otimes_{G_j^{\red}}\CONCR(\mu_{ji})$ and $\CONCR(\mu_{ki})$ are only isomorphic, not equal.
Furthermore, the isomorphism $\boldgamma_i\colon G_i\Rightarrow\CONCR(\mu_{ii})$ 
for each $U_i\in\calO(\calU)$ provides us with a kind
of base point for each $\CONCR(\mu_{ii})$, picking out a unit element $\gamma_i(e_i^{\red})$
for the $G_i^{\red}$-torsor $\CONCR(\mu_{ii})$ (where we write $e_i^{\red}$ for
the unit element of $G_i^{\red}$).

The pseudofunctor $\CONCR$ will form one layer of our atlas, giving the information about the
structure on the quotient space level and the information about the reduced structure underlying
the (possibly) non-reduced one.   However, as discussed above and seen in our examples,
when the action is not effective we may need more atlas embeddings in order to have our atlas
correctly encode the isotropy of the codomain charts.    We will therefore create another layer
of bimodules to our atlas embeddings, which we will call the bimodules of \emph{abstract embeddings}.  
These will be atlas bimodules for the full structure groups;  that is, the action will be free and
transitive for the entire codomain structure group $G_j$, and free for the entire  domain
structure group $G_i$ (and not just the reduced groups as above).
For every pair of  charts $\tU_i,\,\tU_j$ with $U_i\subseteq U_j$ there will be a surjection 
from the bimodule of abstract embeddings to the bimodule of 
concrete embeddings.  So each abstract embedding will correspond to a concrete one,  
but some bimodules of abstract embeddings may be larger, indicating an
ineffective action. 

\begin{defin}\label{defin-02}
Let $X$ be a paracompact, second countable, Hausdorff topological space.   An \emph{orbifold atlas}
of dimension $n$ for $X$, denoted $\mathfrak{U}$, consists of the datum of:

\begin{enumerate}[(1)]
 \item a collection $\calU=\{(\tU_i,G_i,\rho_i,\pi_i)\}_{i\in I}$ of orbifold charts, of
  dimension $n$, connected and simply connected, such that the  reduced charts $\{(\tU_i, G_i^{\red},\pi_i)\}_{i\in I}$ form
  a Satake atlas for $X$; we denote by $(\CONCR,\gamma)\colon\calO(\calU)\to\Groupprof$ the
  induced pseudofunctor as in Lemma~\ref{lem-06};
 
 \item a pseudofunctor
  
  \[{\ABST}:\calO(\calU)\longrightarrow\Groupprof\]
  such that for each $i\in I$, $\ABST(U_i)=G_i$ 
  and for each $\mu_{ji}$ in $\calO(\calU)$, ${\ABST}(\mu_{ji})$ is an atlas bimodule
  $G_i\longarrownot\longrightarrow G_j$, (i.e., the left action of $G_j$ is free and transitive and the right
  action of $G_i$ is free).   We denote the pseudofunctor composition isomorphisms of $\ABST$ by
  
  \[\boldalpha_{kji}\colon{\ABST}(\mu_{kj})\circ{\ABST}(\mu_{ji})\Longrightarrow{\ABST}(\mu_{ki})\]
  for each composable pair of arrows $U_i\stackrel{\mu_{ji}}{\to}U_j\stackrel{\mu_{kj}}{\to}U_k$
  in $\calO(\calU)$ and the identity isomorphisms by
  
  \[\boldalpha_i\colon\mathsf{G}_i\Longrightarrow\ABST(\mu_{ii})\quad\mbox{for each }i\in I;\]

 \item an oplax transformation $\boldrho=(\{\boldrho_i\}_{i\in I}, 
  \{\rho_{ji}\}_{i,j\in I, U_i\subseteq U_j})\colon\ABST\Rightarrow\CONCR$: Recall that each $\rho_i$
  is a group homomorphism from $G_i$ to $G_i^{\red}$, hence it induces a bimodule $\boldrho_i:
  G_i \longarrownot\longrightarrow G_i^{\red}$ (as in Section~\ref{SS:gphom}).
  We require that these $\boldrho_i$ be the components of an
  oplax transformation $\boldrho\colon \ABST\Rightarrow \CONCR$.   So, for each arrow
  $U_i\stackrel{\mu_{ji}}{\longrightarrow} U_j$ in $\calO(\calU)$, there is a map of bimodules,
  
  \[\rho_{ji}:\boldrho_j\otimes_{G_j}{\ABST}(\mu_{ji})\Longrightarrow
  \CONCR(\mu_{ji})\otimes_{G_i^{\red}} \boldrho_i,\]
  as in
  
  \[
  \begin{tikzpicture}[xscale=1.3,yscale=-0.8]
    \node (A0_0) at (0, 0) {$G_i$};
    \node (A0_2) at (2, 0) {$G_j$};
    \node (A2_0) at (0, 2) {$G_i^{\red}$};
    \node (A2_2) at (2, 2) {$G_j^{\red}$};
    
    \node (B1_1) [rotate=315] at (0.9, 1) {$\Downarrow$};
    \node (A1_1) at (1.2, 1) {$\scriptstyle{\rho_{ji}}$};
    
    \path (A2_0) edge [->] node [auto,swap] {$\scriptstyle{\CONCR(\mu_{ji})}$} (A2_2);
    \slashing{\path (A2_0)  -- node [] {$\scriptscriptstyle{\slash}$} (A2_2);}{}
    \path (A0_0) edge [->] node [auto] {$\scriptstyle{\ABST(\mu_{ji})}$} (A0_2);
    \slashing{\path (A0_0)  --node [] {$\scriptscriptstyle{\slash}$} (A0_2);}{}
    \path (A0_2)  edge [->]node [auto] {$\scriptstyle{\boldrho_j}$} (A2_2);
    \slashing{\path (A0_2)  --node [rotate=90]{$\scriptscriptstyle{\slash}$} (A2_2);}{}
    \path (A0_0) edge [->] node [auto,swap] {$\scriptstyle{\boldrho_i}$} (A2_0);
    \slashing{\path (A0_0)  --node [rotate=90] {$\scriptscriptstyle{\slash}$} (A2_0);}{}
  \end{tikzpicture}
  \]
   
  We further require that:

  \begin{enumerate}
    \item  the $\rho_{ji}$ are surjective maps of bimodules;
    \item  \label{D:transitivity} (transitivity on the kernel) whenever
      $\rho_{ji}(e_j^{\red}\otimes\lambda)=\rho_{ji}(e_j^{\red}\otimes \lambda')$ for
      $\lambda,\lambda'\in{\ABST}(\mu_{ji})$, there is an element $g\in G_i$ such that
      $\lambda\cdot g=\lambda'$ (here $e_j^{\red}$ is the identity element of $G_j^{\red}$). 
 \end{enumerate}
\end{enumerate}

For simplicity, we may denote the atlas $\mathfrak{U}$ by $(X,\calU,\ABST,\boldrho)$ or simply 
$(\calU,\ABST,\boldrho)$ if $X$ is clear from the context.
\end{defin}

\begin{Notation}
As in the previous definition, we will use the notation $e_i$ for the unit element of $G_i$ and
$e_i^{\red}$ for the unit element of $G_i^{\red}$.  However, when the group is clear from
the context we will simply write $e$ for the unit element.
\end{Notation}

\begin{rmk}\label{compatibility}
The 2-cell components of an oplax transformation between pseudofunctors are required to be
compatible with the structure cells of the pseudofunctors.   Since we will use this property
several times in the rest of this paper, we spell it out explicitly for the oplax transformation
$\boldrho$.

For any composable pair of arrows $U_i\stackrel{\mu_{ji}}{\rightarrow}U_j\stackrel{\mu_{kj}}{\to}
U_k$ in $\calO(\calU)$,

\[
\begin{tikzpicture}[xscale=2.3,yscale=-1.4]
    \node (A0_1) at (1, 0) {$G_j$};
    \node (A1_2) at (2, 1) {$G_k$};
    \node (A3_0) at (0, 2.5) {$G_i^{\red}$};
    \node (A3_2) at (2, 2.5) {$G_k^{\red}$};
    \node (A1_0) at (0, 1) {$G_i$};
    
    \node (A1_1) at (1, 0.6) {$\Downarrow\,\scriptstyle{\boldalpha_{kji}}$};
    \node (A2_1) [rotate=315] at (0.9, 1.75) {$\Downarrow$};
    \node (K2_1) at (1.1, 1.75) {$\scriptstyle{\rho_{ki}}$};

    \def \z {-0.8}
    \node (C0_0) at (3+\z/2, 0.5) {$\equiv$};
    \node (A0_5) at (5+\z, 0) {$G_j$};
    \node (A1_4) at (4+\z, 1) {$G_i$};
    \node (A1_6) at (6+\z, 1) {$G_k$};
    \node (A2_5) at (5+\z, 1.5) {$G_j^{\red}$};
    \node (A3_4) at (4+\z, 2.5) {$G_i^{\red}$};
    \node (A3_6) at (6+\z, 2.5) {$G_k^{\red}$};

    \node (A0_4) at (4.5+\z, 1) {$\Downarrow\,\scriptstyle{\rho_{ji}}$};
    \node (A0_6) at (5.5+\z, 1) {$\Downarrow\,\scriptstyle{\rho_{kj}}$};
    \node (A3_5) at (5+\z, 2.1) {$\Downarrow\,\scriptstyle{\gamma_{kji}}$};

    \node (B0_0) at (4.4+\z, 1.75) {$\scriptstyle{\CONCR(\mu_{ji})}$};
    \slashing{\path (A3_4) edge [->]node [rotate=25] {$\scriptscriptstyle{\slash}$} (A2_5);}
      {\path (A3_4) edge [->]node [rotate=25] {} (A2_5);}
    \node (B1_1) at (5.6+\z, 1.75) {$\scriptstyle{\CONCR(\mu_{kj})}$};
    \slashing{\path (A2_5) edge [->]node [rotate=335] {$\scriptscriptstyle{\slash}$} (A3_6);}
      {\path (A2_5) edge [->]node [rotate=335] {} (A3_6);}
    \path (A0_5) edge [->] node [auto] {$\,\scriptstyle{\boldrho_j}$} (A2_5);
       \slashing{\path (A0_5) -- node [rotate=90] {$\scriptscriptstyle{\slash}$} (A2_5);}{}
    \path (A1_4) edge [->] node [auto,swap] {$\scriptstyle{\boldrho_i}\,$} (A3_4);
       \slashing{\path (A1_4) -- node [rotate=90] {$\scriptscriptstyle{\slash}$} (A3_4);}{}
    \path (A0_1) edge [->] node [auto] {$\scriptstyle{\ABST(\mu_{kj})}$} (A1_2);
       \slashing{\path (A0_1) -- node [rotate=335] {$\scriptscriptstyle{\slash}$} (A1_2);}{}
    \path (A0_5) edge [->] node [auto] {$\scriptstyle{\ABST(\mu_{kj})}$} (A1_6);
       \slashing{\path (A0_5) -- node [rotate=335] {$\scriptscriptstyle{\slash}$} (A1_6);}{}
    \path (A1_0) edge [->] node [auto,swap] {$\scriptstyle{\boldrho_i}\,$} (A3_0);
       \slashing{\path (A1_0) -- node [rotate=90] {$\scriptscriptstyle{\slash}$} (A3_0);}{}
    \path (A1_0) edge [->] node [auto] {$\scriptstyle{\ABST(\mu_{ji})}$} (A0_1);
       \slashing{\path (A1_0) -- node [rotate=25] {$\scriptscriptstyle{\slash}$} (A0_1);}{}
    \path (A1_4) edge [->] node [auto] {$\scriptstyle{\ABST(\mu_{ji})}$} (A0_5);
       \slashing{\path (A1_4) -- node [rotate=25] {$\scriptscriptstyle{\slash}$} (A0_5);}{}
    \path (A1_2) edge [->] node [auto] {$\,\scriptstyle{\boldrho_k}$} (A3_2);
       \slashing{\path (A1_2) -- node [rotate=90] {$\scriptscriptstyle{\slash}$} (A3_2);}{}
    \path (A1_6) edge [->] node [auto] {$\,\scriptstyle{\boldrho_k}$} (A3_6);
       \slashing{\path (A1_6) --node [rotate=90] {$\scriptscriptstyle{\slash}$} (A3_6);}{}
    \slashing
      {\path (A3_0.350) -- node [auto,swap] {$\scriptstyle{\CONCR(\mu_{ki})}$} (A3_2.190);
        \path (A3_0) edge [->] node [] {$\scriptscriptstyle{\slash}$} (A3_2);}
      {\path (A3_0) edge [->] node [auto,swap] {$\scriptstyle{\CONCR(\mu_{ki})}$} (A3_2);}
    \slashing
       {\path (A3_4.350) -- node [auto,swap] {$\scriptstyle{\CONCR(\mu_{ki})}$} (A3_6.190);
         \path (A3_4) edge [->] node [] {$\scriptscriptstyle{\slash}$} (A3_6);}
       {\path (A3_4) edge [->] node [auto,swap] {$\scriptstyle{\CONCR(\mu_{ki})}$} (A3_6);}
    \slashing
      {\path (A1_0.350) -- node [auto,swap] {$\scriptstyle{\ABST(\mu_{ki})}$} (A1_2.190);
        \path (A1_0) edge [->] node [] {$\scriptscriptstyle{\slash}$} (A1_2);}
      {\path (A1_0) edge [->] node [auto,swap] {$\scriptstyle{\ABST(\mu_{ki})}$} (A1_2);}
\end{tikzpicture}
\]

For any $U_i$ in $\calO(\calU)$,

\[
\begin{tikzpicture}[xscale=1.5,yscale=-0.9]
    \def \z {-0.7}
    \node (A0_0) at (0, 0) {$G_i$};
    \node (A0_2) at (2, 0) {$G_i$};
    \node (A0_6) at (6, 0+\z) {$G_i$};
    \node (A1_3) at (3, 0.8) {$\equiv$};
    \node (A2_0) at (0, 2) {$G_i^{\red}$};
    \node (A2_2) at (2, 2) {$G_i^{\red}$};
    \node (A2_4) at (4, 2+\z) {$G_i^{\red}$};
    \node (A2_6) at (6, 2+\z) {$G_i^{\red}$};
    \node (A0_4) at (4, -0.7) {$G_i$};
    
    \node (A0_1) at (1, -0.35) {$\Downarrow\,\scriptstyle{\boldalpha_i}$};
    \node (A2_5) at (5, 2.35+\z) {$\Downarrow\,\scriptstyle{\gamma_{i}}$};
    \node (A1_5) [rotate=315] at (4.9, 0.9+\z) {$\Downarrow$};
    \node (K1_5) at (5.2, 0.9+\z) {$\scriptstyle{\hat{\rho}_i}$};
    \node (A2_1) [rotate=315] at (0.9, 1.1) {$\Downarrow$};
    \node (K2_1) at (1.2, 1.1) {$\scriptstyle{\rho_{ii}}$};
    
    \path (A2_0) edge [->] node [auto,swap] {$\scriptstyle{\CONCR(\mu_{ii})}$} (A2_2);
       \slashing{\path (A2_0) -- node [] {$\scriptscriptstyle{\slash}$} (A2_2);}{}
    \path (A0_6) edge [->] node [auto] {$\,\scriptstyle{\boldrho_i}$} (A2_6);
       \slashing{\path (A0_6) -- node [rotate=90] {$\scriptscriptstyle{\slash}$} (A2_6);}{}
    \path (A0_4) edge [->] node [auto] {$\scriptstyle{\mathsf{G}_i}$} (A0_6);
    \path (A0_4) edge [->] node [auto,swap] {$\scriptstyle{\boldrho_i}\,$} (A2_4);
    \slashing{\path (A0_4) -- node [] {$\scriptscriptstyle{\slash}$} (A0_6);}{}
    \slashing{\path (A0_4) -- node [rotate=90] {$\scriptscriptstyle{\slash}$} (A2_4);}{}
    \path (A0_2) edge [->] node [auto] {$\,\scriptstyle{\boldrho_i}$} (A2_2);
       \slashing{\path (A0_2) -- node [rotate=90] {$\scriptscriptstyle{\slash}$} (A2_2);}{}
    \path (A0_0) edge [->] node [auto,swap] {$\scriptstyle{\ABST(\mu_{ii})}$} (A0_2);
      \slashing{\path (A0_0) --node [] {$\scriptscriptstyle{\slash}$} (A0_2);}{}
    \path (A2_4) edge [->] node [auto] {$\scriptstyle{\mathsf{G}_i^{\red}}$} (A2_6);
       \slashing{\path (A2_4) -- node [] {$\scriptscriptstyle{\slash}$} (A2_6);}{}
    \path (A0_0) edge [->] node [auto,swap] {$\scriptstyle{\boldrho_i}\,$} (A2_0);
       \slashing{\path (A0_0) -- node [rotate=90] {$\scriptscriptstyle{\slash}$} (A2_0);}{}
    \node (C) at (1, -1.05) {$\scriptstyle{\mathsf{G}_i}$};
    \slashing
       {\path (A0_0) edge [->,bend right=60]node [] {$\scriptscriptstyle{\slash}$} (A0_2);}
       {\path (A0_0) edge [->,bend right=60]node [] {} (A0_2);}
    \node (D) at (5, 3.2+\z) {$\scriptstyle{\CONCR(\mu_{ii})}$};
    \slashing
       {\path (A2_4) edge [->,bend left=60]node [] {$\scriptscriptstyle{\slash}$} (A2_6);}
       {\path (A2_4) edge [->,bend left=60]node [] {} (A2_6);}
\end{tikzpicture}
\]
where $\hat{\rho}_i$ is the transformation determined by
$\hat{\rho}_i(e\otimes g):=\boldrho_i(g)\otimes e$.
\end{rmk}

\begin{rmk}
Satake did not require in his original definition of orbifold atlas that there be an atlas
embedding  (a Satake embedding in our terminology) $\tU\hookrightarrow \tV$ whenever $U\subseteq V$
in the underlying space.   Instead, he required that for each point $x\in U\cap V$ there be a
chart $W$ containing $x$ with chart embeddings $\tW\hookrightarrow\tU$ and $\tW\hookrightarrow\tV$.
However, as shown in~\cite{MP}, when we require that the charts be connected and simply connected,
this condition implies that there is a Satake embedding $\tU\hookrightarrow \tV$ whenever
$U\subseteq V$ in the underlying space.
 
So when we set up the framework for our definition of orbifold atlas we had the choice to work
with more general charts and a slightly weaker compatibility condition or to work with simply
connected charts.   It is more work to prove in this context, but if one were given an atlas
with only the weaker compatibility condition (so we would take $\calO(\calU)$ to be
a subcategory of $\calO(X)$ which is not necessarily a full subcategory,  but still satisfies a weaker
condition) where the charts were connected and simply connected, then it would be possible
to extend $\CONCR$ and $\ABST$ to pseudofunctors on the full subcategory and extend $\boldrho$
to a pseudonatural transformation between them.   This result can be found in
forthcoming work by Sibih~\cite{thesis}.   In order to make the arguments on the relationship
between orbifold atlases and groupoids run more smoothly we have chosen to use the slightly
stronger compatibility condition as stated in the definition above, but all results still
work when one works with the weaker condition.
\end{rmk}

We now show the connection between our new atlas definition and the previously defined
version discussed in Section~\ref{sec-2}.

\begin{lma}\label{tilde-l}
Each element $\lambda\in{\ABST}(\mu_{ji})$ defines a concrete chart embedding $\trho_{ji}(\lambda)
= \tlambda \colon \tU_i\to\tU_j$ and a group homomorphism
$\ell_{\lambda}\colon G_i\ \to  G_j$ such that  $\ell_{\lambda}$ restricts to an isomorphism on
the kernels and 
$\tlambda$ is equivariant with respect to $\ell_{\lambda}$ in the sense that

\[\tlambda(\rho_i(g)\cdot x)=\rho_j(\ell_{\lambda}(g))\cdot\tlambda(x),\]
for all $x\in\tU_i$ and $g\in G_i$.
\end{lma}

\begin{proof}
The map of bimodules
$\rho_{ji}\colon \boldrho_j\otimes_{G_j}\ABST(\mu_{ji})\Rightarrow\CONCR(\mu_{ji})
\otimes_{G_i^{\red}}\boldrho_i$ is a surjective map of sets which is equivariant with respect
to the actions of $G_i$ and $G_j^{\red}$.
The underlying set $S$ of $\boldrho_j\otimes_{G_j}{\ABST}(\mu_{ji})$ consists of elements of
the form $g\otimes\lambda$ where $g\in G_j^{\red}$ and $\lambda\in {\ABST}(\mu_{ji})$; and 
for any $k\in G_j$ we have that $g\cdot \rho_j(k)\otimes \lambda=g\otimes k\cdot\lambda$.
Since $\rho_j$ is surjective onto $G_j^{\red}$, this means that each element of $S$ can be written
in the form $e\otimes\lambda$.   However, this representation is generally not unique:
whenever $k\in \mbox{Ker}(\rho_j)$, we have that $e\otimes k\cdot\lambda=e\otimes \lambda$.

Similarly, each element of the underlying set of ${\CONCR}(\mu_{ji})\otimes_{G_i^{\red}}\boldrho_i$
can be represented as $\mu\otimes e$ for some $\mu\in {\CONCR}(\mu_{ji})$. 
Here the representation is unique: $\mu\otimes e=\mu'\otimes e$ if and only if $\mu=\mu'$.
This allows us to define a map $\trho_{ji}\colon {\ABST}(\mu_{ji})\to {\CONCR}(\mu_{ji})$ by

\begin{equation}\label{tilderho}
\rho_{ji}(e\otimes\lambda)=\trho_{ji}(\lambda)\otimes e.
\end{equation}
We write $\tlambda$ for $\trho_{ji}(\lambda)$.
So $\tlambda$ is the unique element of $\CONCR(\mu_{ji})$ such that
$\rho_{ji}(e\otimes\lambda)=\tlambda\otimes e$. 

Note that if $k\in \mbox{Ker}(\rho_j)$, we have that $\tlambda=\widetilde{k\cdot\lambda}$.
Also, the map $\trho_{ji}$ is equivariant with respect to $\rho_i$ and $\rho_j$:
For $g\in G_i$,

\begin{eqnarray*}
\widetilde{\lambda\cdot g}\otimes e &=& \rho_{ji}(e\otimes(\lambda\cdot g)) \\
& = & \rho_{ji}((e\otimes\lambda)\cdot g) \\
& = & \rho_{ji}(e\otimes\lambda)\cdot g \\
& = & (\tlambda\otimes e)\cdot g \\
& = & \tlambda\otimes g \\
& = &(\tlambda\cdot{\rho}_i(g))\otimes e
\end{eqnarray*}
so
\[\widetilde{\lambda\cdot g}=\tlambda\cdot{\rho}_i(g).\]
A similar argument shows that  for $h\in G_j$,
%
\[\widetilde{h\cdot\lambda}=\rho_j(h)\cdot\tlambda.\]

Now let $\lambda\in {\ABST}(\mu_{ji})$ be arbitrary and let $\ell_{\lambda}\colon G_i\to G_j$ 
be the induced group homomorphism defined by $\lambda\cdot g=\ell_{\lambda}(g)\cdot \lambda$ as in
Lemma~\ref{lem-01}.   Note that $\ell_{\lambda}$ sends the kernel of $\rho_i$ to the kernel of
$\rho_j$:   If $k\in \mbox{Ker}(\rho_i)$, then $\trho_{ji}(\lambda\cdot k)=\trho_{ji}(\lambda)
\cdot\rho_i(k)=\trho_{ji}(\lambda)$.  Hence, $\rho_j(\ell_{\lambda}(k))\cdot
\trho_{ji}(\lambda)=\trho_{ji}(\ell_{\lambda}(k)\cdot \lambda)=
\trho_{ji}(\lambda)$.  So $\rho_j(\ell_{\lambda}(k))\in\mbox{Ker}(\rho_j)$.

Furthermore, $\ell_{\lambda}$ is also surjective on the kernels: Let $h\in\mbox{Ker}(\rho_j)$.
Then $\rho_{ji}(e\otimes h\cdot\lambda)=\rho_{ji}(\rho_j(h)\otimes \lambda)=
\rho_{ji}(e\otimes\lambda)$,
so by condition (\ref{D:transitivity}) of Definition~\ref{defin-02}, there is an element
$g\in G_i$ such that $h\cdot\lambda=\lambda\cdot g$.  Hence, $\ell_{\lambda}(g)=h$.
We check that $g\in\mbox{Ker}(\rho_i)$:
\begin{eqnarray*}
\trho_{ji}(\lambda)\cdot\rho_i(g)&=&\trho_{ji}(\lambda\cdot g) \\
& = & \trho_{ji}(h\cdot\lambda) \\
& = & \rho_j(h)\cdot\trho_{ji}(\lambda) \\
& = & \trho_{ji}(\lambda).
\end{eqnarray*}
Since the action of $G_i^{\red}$ is free on $\tU_i$, this implies that $\rho_{i}(g)=e$.

Since the homomorphism $\ell_{\lambda}$ is injective by Lemma~\ref{lem-01} we conclude
that it restricts to an isomorphism on the kernels.
\end{proof}

\begin{rmks}\label{D:rmks}
\begin{enumerate}
 \item\label{D:prescomp}
  The operation $(\widetilde{\phantom{ooo}})$ of taking the underlying concrete embedding commutes
  with composition in the sense that for $\lambda\in\ABST(\mu_{ji})$ and $\nu\in\ABST(\mu_{kj})$,
  $\widetilde{\left(\boldalpha_{kji}(\nu\otimes\lambda)\right)}
  =\gamma_{kji}(\tilde{\nu},\tilde{\lambda})=\tnu\circ\tlambda$. We can
  see this as follows.   The compatibility of the $\rho_{ji}$ with the $\boldalpha_{kji}$ and
  $\gamma_{kji}$ from Remark~\ref{compatibility} gives us that $\rho_{ki}(e\otimes\boldalpha_{kji}
  (\nu\otimes\lambda))=\gamma_{kji}(\tnu\otimes\tlambda)\otimes e$.
  So by identity \eqref{tilderho} we get
  $\widetilde{\left(\boldalpha_{kji}(\nu\otimes\lambda)\right)}=\gamma_{kji}(\tnu\otimes
  \tlambda)$.   But $\gamma_{kji}$ is just the operation of taking the composition, so we get
  $\widetilde{\left(\boldalpha_{kji}(\nu\otimes\lambda)\right)}=\tnu\circ\tlambda$, as claimed.
  By definition of $\trho_{ki}$, this is equivalent to saying that
  $\trho_{ki}(\boldalpha_{kji}(\nu\otimes\lambda))=\tnu\circ \tlambda$.
  
 \item\label{D:surj}
  Since each $\rho_{ji}$ is surjective, the map $\trho_{ji}$ 
  sending $\lambda\in \ABST(\mu_{ji})$ to $\tlambda\in\CONCR(\mu_{ji})$ is also surjective.
  
 \item\label{D:transitivity-rephrased} Using condition (\ref{D:transitivity}) in
  Definition~\ref{defin-02} and identity \eqref{tilderho}, if $\trho_{ji}(\lambda)=
  \trho_{ji}(\lambda')$, then there is an element $g\in G_i$ such that $\lambda\cdot g = \lambda'$.
\end{enumerate}
\end{rmks}

\begin{rmks}  Any Satake atlas gives rise to a canonical atlas as in Definition~\ref{defin-02}:
we simply take $\ABST(\mu_{ji})=\CONCR(\mu_{ji})$ and $\rho_{ji}$ is defined by
$\rho_{ji}(e_j\otimes\lambda)=\lambda\otimes e_i$.   It is possible that the $\rho_{ji}$ may be given by different isomorphisms, instead of the ones just described.
However, each effective atlas is isomorphic to one coming from a Satake atlas via a
canonical isomorphism.
\end{rmks}



\begin{exa}\label{E:z3p2}
In order to make these ideas more concrete, we briefly describe the ineffective orbifold atlas
structures for the two $\bbZ/3$ orbigroupoids $\calG$ and $\calH$ described in Example~\ref{E:z3}.  

Both orbifolds consist of completely ineffective actions on the circle $S^1$.
So for both atlas structures, we may take four charts to cover $S^1$,
one for the upper semicircle and one for the lower (here denoted $U_1$ and $U_2$),
and two smaller charts embedding into the overlaps on the left and right sides
(denoted $U_3$ and $U_4$).

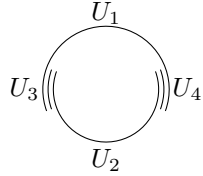
\[
\begin{tikzpicture}[xscale=0.7,yscale=0.7]
    \def \radiusmax {1.2}
    \def \radiusmin {1.0}mu
  
    \def \radius {1.1}
  
    \node (A1_1) at (0, \radiusmax + 0.25) {$U_1$};
    \node (A2_2) at (0, -\radiusmax - 0.15) {$U_2$};
    \node (A3_3) at (-\radiusmax - 0.35, 0) {$U_3$};
    \node (A4_2) at (\radiusmax + 0.35, 0) {$U_4$};
    
    \draw[black] ([shift=(-20:\radiusmax cm)]0,0) arc (-20:200:\radiusmax cm);
    \draw[black] ([shift=(160:\radiusmin cm)]0,0) arc (160:380:\radiusmin cm);
    \draw[black] ([shift=(-20:\radius cm)]0,0) arc (-20:20:\radius cm);
    \draw[black] ([shift=(160:\radius cm)]0,0) arc (160:200:\radius cm);
\end{tikzpicture}
\]
  
For each chart $U_i$, we have $G_{U_i} = \bbZ/3$ and $G_{U_i}^{\red} = \{e\}$.
We also have inclusions
$\lambda_{13}: U_3 \hookrightarrow U_1, \lambda_{23}: U_3 \hookrightarrow U_2,  \lambda_{14}: U_4 \hookrightarrow U_1$ and
$ \lambda_{24}:U_4 \hookrightarrow U_2$.   So for each inclusion $\mu_{ji}:U_i\hookrightarrow U_j$,
we need to define a module $\mathsf{M}_{ji}$ (with compatible left free, transitive
action by $G_{U_j}$ and right free action by $G_{U_i}$)
and a map of bimodules $\rho_{ji}$ as follows:
 
\begin{equation}\label{eq-03}
\begin{tikzpicture}[xscale=2.5,yscale=-0.8]
    \node (A0_0) at (0, 0) {$\bbZ/3= \{e, \omega_i, \omega_i^2\}$};
    \node (A0_2) at (2, 0) {$\bbZ/3= \{e, \omega_j, \omega_j^2\} $};
    \node (A2_0) at (0, 2) {$\{e\}$};
    \node (A2_2) at (2, 2) {$\{e\}$};
    
    \node (B1_1) [rotate=315] at (0.9, 1) {$\Downarrow$};
    \node (A1_1) at (1.1, 1) {$\scriptstyle{\rho_{ji}}$};
    
    \path (A0_0) edge [->]node [auto,swap] {$\scriptstyle{\boldrho_i}$} (A2_0);
    \slashing{\path (A0_0)  --node [rotate=90] {$\scriptscriptstyle{\slash}$} (A2_0);}{}

    \path (A2_0) edge [->]node [auto,swap] {$\scriptstyle{\CONCR(\mu_{ji})=\{\lambda_{ji}\}}$} (A2_2);
    \slashing{\path (A2_0)  -- node [] {$\scriptscriptstyle{\slash}$} (A2_2);}{}

    \path (A0_2) edge [->]node [auto] {$\scriptstyle{\,\boldrho_j}$} (A2_2);
    \slashing{\path (A0_2)  --node [rotate=90]{$\scriptscriptstyle{\slash}$} (A2_2);}{}

    \path (A0_0) edge [->]node [auto] {$\scriptstyle{\ABST(\mu_{ji})=\mathsf{M}_{ji}}$} (A0_2);
    \slashing{\path (A0_0)  --node [] {$\scriptscriptstyle{\slash}$} (A0_2);}{}
\end{tikzpicture}
\end{equation}

Using the description given in the proof of Lemma~\ref{tilde-l}, in this case
both the source and the target of $\rho_{ji}$ are a module with one element
and trivial actions, so we define $\rho_{ji}$ as the unique isomorphism.     Since the left
action of $G_{U_j}$ must be free and transitive,  each $M_{ji}$ must consist of three objects,
say $a_{ji},b_{ji},c_{ji}$.\\

In order to create  an atlas for the orbifold $\calG$ represented by the groupoid with three disjoint
circles in the arrow space,  we define the various bimodules $M_{ji}$ by giving the following actions:

\begin{equation}\label{eq:a}
\begin{tikzpicture}[xscale=1.5,yscale=-1.2]
    \node (B0_0) at (-2, 0) {$\textrm{left multiply by }\omega_j$};

    \node (A0_0) at (0, 0) {$a_{ji}$};
    \node (A0_1) at (1, 0) {$b_{ji}$};
    \node (A0_2) at (2, 0) {$c_{ji}$};
    \path (A0_0) edge [->]node [auto] {$\scriptstyle{}$} (A0_1);
    \path (A0_1) edge [->]node [auto] {$\scriptstyle{}$} (A0_2);
    \path (A0_2) edge [->,bend left=35]node [auto] {$\scriptstyle{}$} (A0_0);

    \node (B1_0) at (-2, 1) {$\textrm{right multiply by }\omega_i$};

    \node (A1_0) at (0, 1) {$a_{ji}$};
    \node (A1_1) at (1, 1) {$b_{ji}$};
    \node (A1_2) at (2, 1) {$c_{ji}$};
    \path (A1_0) edge [->]node [auto] {$\scriptstyle{}$} (A1_1);
    \path (A1_1) edge [->]node [auto] {$\scriptstyle{}$} (A1_2);
    \path (A1_2) edge [->,bend left=35]node [auto] {$\scriptstyle{}$} (A1_0);
\end{tikzpicture}
\end{equation}

In contrast, to create an atlas for the orbifold  $\calH$ represented by the groupoid with only two
disjoint circles in the arrow space, we set $M_{13}, M_{14} $ and $M_{23}$ to be identical
to the module described in \eqref{eq:a}.   However, for the last inclusion we
define $M_{24}$ with action given by
\begin{equation}\label{eq:b}
\begin{tikzpicture}[xscale=1.5,yscale=-1.2]
    \node (B0_0) at (-2, 0) {$\textrm{left multiply by }\omega_j$};

    \node (A0_0) at (0, 0) {$a_{ji}$};
    \node (A0_1) at (1, 0) {$b_{ji}$};
    \node (A0_2) at (2, 0) {$c_{ji}$};
    \path (A0_0) edge [->]node [auto] {$\scriptstyle{}$} (A0_1);
    \path (A0_1) edge [->]node [auto] {$\scriptstyle{}$} (A0_2);
    \path (A0_2) edge [->,bend left=35]node [auto] {$\scriptstyle{}$} (A0_0);

    \node (B1_0) at (-2, 1) {$\textrm{right multiply by }\omega_i$};

    \node (A1_0) at (0, 1) {$a_{ji}$};
    \node (A1_1) at (1, 1) {$b_{ji}$};
    \node (A1_2) at (2, 1) {$c_{ji}$};
    \path (A1_1) edge [->]node [auto] {$\scriptstyle{}$} (A1_0);
    \path (A1_2) edge [->]node [auto] {$\scriptstyle{}$} (A1_1);
    \path (A1_0) edge [->,bend right=35]node [auto] {$\scriptstyle{}$} (A1_2);
\end{tikzpicture}
\end{equation}

It is easy to see that all required compatibilities hold here.  
Note that these atlases can actually be constructed from the groupoid description
of Example~\ref{E:z3}; we will discuss this process in Section~\ref{S:atlastogroupoid}.  
\end{exa}

\subsection{Further Results}
This section lists some results that generalize properties of the concrete embeddings to
the abstract embeddings in an orbifold atlas.  

\begin{lma}\label{left-cancel}
For $\lambda\in\ABST(\mu_{kj})$ and $\mu,\mu'\in\ABST(\mu_{ji})$,
if $\lambda\otimes\mu=\lambda\otimes\mu'$ then $\mu=\mu'$.
\end{lma}

\begin{proof}
If $\lambda\otimes\mu=\lambda\otimes\mu'$ then $\boldalpha_{kji}(\lambda\otimes\mu)=
\boldalpha_{kji}(\lambda\otimes\mu')$ and hence, $\tlambda\circ\tmu=\tlambda\circ\tmu'$
by Remark~\ref{D:rmks}(\ref{D:prescomp}).   Since $\tlambda$ is injective, this implies
that $\tmu=\tmu'$, i.e., $\trho_{ji}(\mu)=\trho_{ji}(\mu')$.
By Remark~\ref{D:rmks}(\ref{D:transitivity-rephrased}), there is an element
$g\in G_i$ such that $\mu\cdot g=\mu'$.   Hence, $\lambda\otimes\mu=\lambda\otimes\mu'=
\lambda\otimes(\mu\cdot g)=(\lambda\otimes \mu)\cdot g$.  However, $G_i$ acts freely
on $\ABST(\mu_{ji})$ by condition (2) of Definition~\ref{defin-02}.   Hence, $g=e_i$ and $\mu=\mu'$.
\end{proof}

\begin{lma}\label{right-cancel}
For $\lambda,\lambda'\in\ABST(\mu_{kj})$ and $\mu\in\ABST(\mu_{ji})$,
if $\lambda\otimes\mu=\lambda'\otimes\mu$ then $\lambda=\lambda'$.
\end{lma}

\begin{proof}
Since $G_k$ acts transitively on $\ABST(\mu_{kj})$, there is an element $g\in G_k$ such that
$g\cdot \lambda=\lambda'$.
This implies that $g\cdot\boldalpha_{kji}(\lambda\otimes\mu)=\boldalpha_{kji}(\lambda\otimes\mu)$.
Since the action of $G_k$ on $\ABST(\mu_{ki})$ is free, this implies that $g=e_k$ and hence,
$\lambda=\lambda'$.
\end{proof}

\begin{lma}\label{interpolation}
If $U_i\le U_j\le U_k$ in $\calO(\calU)$ and $\lambda\in\ABST(\mu_{ki})$, $\nu\in\ABST(\mu_{kj})$
with $\tlambda(\tU_i)\cap\tnu(\tU_j)\neq \varnothing$, then there is a unique
$\kappa\in \ABST(\mu_{ji})$ such that $\boldalpha_{kji}(\nu\otimes\kappa)=\lambda$.
\end{lma}

\begin{proof}
First note that by Lemma~\ref{factorization}
there is a unique element $\theta\in \CONCR(\mu_{ji})$ such that $\tnu\circ\theta=
\tlambda$.  By Remark~\ref{D:rmks}(\ref{D:surj}), there is an element $\kappa'\in \ABST(\mu_{ji})$
such that $\tkappa'=\theta$.  
Then by Remark~\ref{D:rmks}(\ref{D:prescomp}) we have
$\trho_{ji}(\boldalpha_{kji}(\nu\otimes\kappa'))=\tnu\circ\tkappa'=\tnu\circ\theta=
\tlambda=\trho_{ji}(\lambda)$.  So using Remark~\ref{D:rmks}(\ref{D:transitivity-rephrased})
there is an element $g\in G_i$ such that 
$\boldalpha_{kji}(\nu\otimes(\kappa'\cdot g))=\boldalpha_{kji}(\nu\otimes\kappa')\cdot g=\lambda$.
Let $\kappa=\kappa'\cdot g$.  The fact that $\kappa$ is unique with this property follows
from Lemma~\ref{left-cancel} together with the fact that $\boldalpha_{kji}$
is an isomorphism by condition (2) of Definition~\ref{defin-02}.
\end{proof}

The strong local compatibility property for concrete embeddings of atlas charts extends to abstract embeddings
in the following way.

\begin{lma}\label{strong_compatibility}
For any charts $\tU_1,\tU_2,\tU_3$ with $U_1\le U_3$ and $U_2\le U_3$ in $\calO(\calU)$ 
with points $x_i\in\tU_i$ for $i=1,2,3$ and abstract embeddings $\lambda_{3i}\in\ABST(\mu_{3i})$
such that $\tlambda_{3i}(x_i)=x_3$ for $i=1,2$, there is a chart $\tU_4$ with a point
$y\in\tU_4$ and abstract embeddings $\kappa_i\in\ABST(\mu_{i4})$
such that $\tkappa_{i4}(y)=x_i$ for $i=1,2$ and $\boldalpha_{314}(\lambda_{31}\otimes
\kappa_{14})=\boldalpha_{324}(\lambda_{32}\otimes\kappa_{24})$.
\end{lma}

\begin{proof}
By Lemma~\ref{strong_comp} there is a chart $\tU_4$ in the atlas $\calU$ with
$U_4\subseteq U_1\cap U_2$, and concrete embeddings $\nu_{i4}\in\CONCR(\mu_{4i})$ for $i=1,2$,
such that $\tlambda_{31}\circ\nu_{14}=\tlambda_{32}\circ\nu_{24}$.   Furthermore, there is
a point $y\in\tU_4$, such that $\nu_{i4}(y)=x_i$ for $i=1,2$.
Since $\trho_{14}$ and $\trho_{24}$ are surjective by Remark~\ref{D:rmks}(\ref{D:surj}),
there are abstract embeddings $\theta_{i4}$ with $\ttheta_{i4}=\nu_{i4}$ for $i=1,2$.   Then 

\begin{eqnarray*}
\trho_{34}(\boldalpha_{314}(\lambda_{31}\otimes\theta_{14})) & = &
   \tlambda_{31}\circ \ttheta_{14} \\
& = &\tlambda_{31}\circ\nu_{14} \\
& = &\tlambda_{32}\circ\nu_{24} \\
& = &\tlambda_{32}\circ\ttheta_{24} \\
& = &\trho_{34}(\boldalpha_{324}(\lambda_{32}\otimes\theta_{24})).
\end{eqnarray*}
By Remark~\ref{D:rmks}(\ref{D:transitivity-rephrased})
there is an element $g\in G_4$ such that $\boldalpha_{314}(\lambda_{31}\otimes
\theta_{14})\cdot g=\boldalpha_{324}(\lambda_{32}\otimes\theta_{24})$ and hence,
$\boldalpha_{314}(\lambda_{31}\otimes\theta_{14}\cdot g)=
\boldalpha_{324}(\lambda_{32}\otimes\theta_{24})$.

If we define $\kappa_{14}=\theta_{14}\cdot g$ and $\kappa_{24}=\theta_{24}$, then we immediately have
that $\tkappa_{24}(y)=x_2$ and $\boldalpha_{314}(\lambda_{31}\otimes
\kappa_{14})=\boldalpha_{324}(\lambda_{32}\otimes\kappa_{24})$.   These properties imply that
\begin{eqnarray*}
\tlambda_{31}\circ\tkappa_{14}(y) & = & \tlambda_{32}\circ\tkappa_{24}(y) \\
& = & \tlambda_{32}(x_2) \\
& = & x_3 \\
& = & \tlambda_{31}(x_1).
\end{eqnarray*}
Since $\tlambda_{31}$ is injective,  $\tkappa_{14}(y)= x_1$.
So $\kappa_{14}$ and $\kappa_{24}$ have the required property.
\end{proof}
  
\section{Constructing a Groupoid from an Orbifold Atlas}\label{sec-5}

Let $\mathfrak{U}$ be an orbifold atlas as in  Definition~\ref{defin-02};
to fix notation, say that we have open sets $U_i$ with inclusions given by $\mu_{ji}: U_i \to U_j$,
and atlas bimodules $\ABST(\mu_{ji}): G_i \longarrownot\longrightarrow G_j$, with pseudofunctor
structure defined by identity isomorphisms $\boldalpha_i: \mathsf{G}_i\Rightarrow \ABST(\mu_{ii})$ for each $U_i$ in $\calO(\calU)$
and composition isomorphisms $\boldalpha_{kji}\colon{\ABST}(\mu_{kj})\otimes{\ABST}(\mu_{ji})
\Rightarrow{\ABST}(\mu_{ki})$ for each composable pair of arrows
$U_i\stackrel{\mu_{ji}}{\to}U_j\stackrel{\mu_{kj}}{\to}U_k$ in $\calO(\calU)$. 
In order to obtain an \'etale groupoid representation of this orbifold, we will first construct
a smooth category $\CU$ for which the arrows do not necessarily have inverses, and then construct
its smooth groupoid of fractions by an internal version of the Gabriel-Zisman construction
as described in~\cite{PSi}.

The smooth category $\CU$ has space of objects defined by the disjoint union of the charts:

\[\CU_0=\coprod_{\tU_i\in\,\calU}\tU_i.\]
The space of arrows $\CU_1$ is constructed using the atlas bimodules $\ABST$, where
each $\ABST(\mu_{ji})$ has the discrete topology:

\[\CU_1=\coprod_{\mu_{ji}\in\calO(\calU)}\ABST(\mu_{ji})\times \tU_i.\]
The source map is defined by the projection $s(\lambda,x)=x\in \tU_i,$
and the target map uses the concrete embedding $\tlambda$ described in Lemma~\ref{tilde-l},
$t(\lambda,x)=\tlambda(x)\in\tU_j$.
The unit map $u\colon \CU_0\to\CU_1$ is given by 
$u(x)=(\boldalpha_i(e_i),x)$ for $x\in\tU_i$.

We define the composition $m\colon \CU_1\times_{\CU_0}\CU_1\to \CU_1$ in  $\CU$ using the
composition isomorphisms $\boldalpha_{kji}$ of $\ABST$:  
Suppose that $(\lambda, x) \in \ABST(\mu_{ji}) \times \tU_i$ and $(\lambda', x')
\in \ABST(\mu_{kj}) \times \tU_j$, such that  $t(\lambda,x)=s(\lambda',x')$; then we know that
$x'=\tlambda(x)$ and we define
\[m((\lambda',x'),(\lambda,x))=(\boldalpha_{kji}(\lambda'\otimes\lambda),x).\]
This gives us a smooth category with structure maps that are \'etale, since they are embeddings
when restricted to any connected component of the space of arrows.  

Next, we want to construct a smooth groupoid $\GU$ from the smooth category $\CU$.  We do this
using a category of fractions construction.  For ordinary categories, the groupoid of fractions 
can be constructed using the Gabriel-Zisman span construction whenever the following two 
conditions hold:

\begin{itemize}
 \item{(O)}  {\it The Ore condition:}  for each cospan of arrows
  \begin{tikzpicture}[baseline, xscale=1.2]
    \node [anchor=base] (A0_0) at (0, 0) {$A$};
    \node [anchor=base] (A0_1) at (1, 0) {$C$};
    \node [anchor=base] (A0_2) at (2, 0) {$B$};
    \path (A0_0) edge [->]node [auto] {$\scriptstyle{f}$} (A0_1);
    \path (A0_2) edge [->]node [auto,swap] {$\scriptstyle{w}$} (A0_1);
  \end{tikzpicture}
  there are arrows to complete this to a commutative square
  
  \[
  \begin{tikzpicture}[xscale=0.7,yscale=-0.7]
    \node (A0_0) at (0, 0) {$D$};
    \node (A0_2) at (2, 0) {$B$};
    \node (A2_0) at (0, 2) {$A$};
    \node (A2_2) at (2, 2) {$C$};
    \path (A0_0) edge [->]node [auto,swap] {$\scriptstyle{v}$} (A2_0);
    \path (A0_0) edge [->]node [auto] {$\scriptstyle{g}$} (A0_2);
    \path (A0_2) edge [->]node [auto] {$\scriptstyle{w}$} (A2_2);
    \path (A2_0) edge [->]node [auto,swap] {$\scriptstyle{f}$} (A2_2);
  \end{tikzpicture}
  \]
 
 \item{(WC)} {\it  Weak cancellability:} for any arrows
 \begin{tikzpicture}[baseline, xscale=1.2]
    \def \z {15}
    \node [anchor=base] (C) at (0, 0) {$C$};
    \node [anchor=base] (B) at (1, 0) {$B$};
    \node [anchor=base] (A) at (2, 0) {$A$};
    \path (C.\z) edge [->]node [auto] {$\scriptstyle{f}$} (B.180-\z);
    \path (C.360-\z) edge [->]node [auto,swap] {$\scriptstyle{g}$} (B.180+\z);
    \path (B) edge [->]node [auto] {$\scriptstyle{h}$} (A);
 \end{tikzpicture} 
  for which $h f = h g$,  there is an arrow $j: D \to C$  such that $fj=gj$.
\end{itemize}
In the Gabriel-Zisman groupoid of fractions, each arrow is represented by a span
\begin{tikzpicture}[baseline, xscale=1]
    \node [anchor=base] (A) at (0, 0) {};
    \node [anchor=base] (B) at (1, 0) {};
    \node [anchor=base] (C) at (2, 0) {};
    \path (B) edge [->]node [auto,swap] {$\scriptstyle{f}$} (A);
    \path (B) edge [->]node [auto] {$\scriptstyle{g}$} (C);
\end{tikzpicture} 
in the original category, and two such spans,
\begin{tikzpicture}[baseline, xscale=1]
  \node [anchor=base] (A) at (0, 0) {};
    \node [anchor=base] (B) at (1, 0) {};
    \node [anchor=base] (C) at (2, 0) {};
    \path (B) edge [->]node [auto,swap] {$\scriptstyle{f_1}$} (A);
    \path (B) edge [->]node [auto] {$\scriptstyle{g_1}$} (C);
\end{tikzpicture} 
and
\begin{tikzpicture}[baseline, xscale=1]
    \node [anchor=base] (A) at (0, 0) {};
    \node [anchor=base] (B) at (1, 0) {};
    \node [anchor=base] (C) at (2, 0) {};
    \node [anchor=base] (D) at (2, 0) {,};
    \path (B) edge [->]node [auto,swap] {$\scriptstyle{f_2}$} (A);
    \path (B) edge [->]node [auto] {$\scriptstyle{g_2}$} (C);
\end{tikzpicture}
represent the same arrow when there is a third span
\begin{tikzpicture}[baseline, xscale=1]
    \node [anchor=base] (A) at (0, 0) {};
    \node [anchor=base] (B) at (1, 0) {};
    \node [anchor=base] (C) at (2, 0) {};
    \path (B) edge [->]node [auto,swap] {$\scriptstyle{h_1}$} (A);
    \path (B) edge [->]node [auto] {$\scriptstyle{h_2}$} (C);
\end{tikzpicture} 
to make the following diagram commute,

\begin{equation}\label{spansofspans}
\begin{tikzpicture}[xscale=1.5,yscale=-0.4]
    \node (A0_1) at (1, 0) {};
    \node (A2_0) at (0, 2) {};
    \node (A2_1) at (1, 2) {};
    \node (A2_2) at (2, 2) {};
    \node (A4_1) at (1, 4) {};
    \path (A2_1) edge [->]node [auto] {$\scriptstyle{h_2}$} (A4_1);
    \path (A4_1) edge [->]node [auto] {$\scriptstyle{f_2}$} (A2_0);
    \path (A4_1) edge [->]node [auto,swap] {$\scriptstyle{g_2}$} (A2_2);
    \path (A0_1) edge [->]node [auto,swap] {$\scriptstyle{f_1}$} (A2_0);
    \path (A2_1) edge [->]node [auto,swap] {$\scriptstyle{h_1}$} (A0_1);
    \path (A0_1) edge [->]node [auto] {$\scriptstyle{g_1}$} (A2_2);
\end{tikzpicture}
\end{equation}
This relation is an equivalence relation and defines a congruence on the arrow structure when
the category satisfies the conditions (O) and (WC).

In our case, we want to ensure that the resulting groupoid carries an induced topological
structure.   Hence, we want to use the analogous construction internal to the category of
topological spaces as described in~\cite{PSi}.   The construction described there starts
with a topological category and gives topological versions of the (O) and (WC) conditions.
It then considers the  space of spans of arrows in $\CU$, constructed as the pullback
$\CU_1\times_{\CU_0}\CU_1$ of $s$ along $s$, and a space encoding diagrams of the form
\eqref{spansofspans}.   This latter space has two obvious projection maps to
$\CU_1\times_{\CU_0}\CU_1$.   Then  $\GU_1$ is obtained as the coequalizer of these projection
maps, and they form its kernel pair.   On the point level this still implies that the relation
described here is an equivalence relation on the space of spans.

We also want the resulting groupoid $\GU$ to inherit the smooth structure of $\CU$.  
However, working out the corresponding conditions inside the category of smooth manifolds is
a bit tricky, as not all pullbacks exist there.
Fortunately, we can work with the  topological version;   we will show that the spaces we obtain
have induced smooth structures that make all structure maps local diffeomorphisms.  

Rather than reviewing the general construction from~\cite{PSi} in detail, we will just focus on 
what it means for our atlas category $\CU$.
The condition corresponding to the (O) condition above requires us to consider the spaces  
of cospans of arrows and of commutative squares  in $\CU$.  

Since the space of arrows $\CU$ is defined by the disjoint union of charts,  the space
$\mbox{ Cospan}\,(\mathcal{C}(\mathfrak{U}))=\CU_1\times_{t,\CU_0,t}\CU_1$ representing the cospans of arrows
is a coproduct of pullbacks:  for each pair  $\lambda\in\ABST(\mu_{ki})$ and
$\xi\in \ABST(\mu_{kj})$, we create the pullback

\[
\begin{tikzpicture}[xscale=0.9,yscale=-0.8]
    \node (A0_0) at (0, 0) {$P(\lambda,\xi)$};
    \node (A0_2) at (2, 0) {$\tU_j$};
    \node (A2_0) at (0, 2) {$\tU_i$};
    \node (A2_2) at (2, 2) {$\tU_k$};
    \path (A0_0) edge [->]node [auto] {$\scriptstyle{}$} (A2_0);
    \path (A0_0) edge [->]node [auto] {$\scriptstyle{}$} (A0_2);
    \path (A0_2) edge [->]node [auto] {$\scriptstyle{\txi}$} (A2_2);
    \path (A2_0) edge [->]node [auto,swap] {$\scriptstyle{\tlambda}$} (A2_2);
\end{tikzpicture}
\]
Then 

\[\mbox{ Cospan}\,(\mathcal{C}(\mathfrak{U})) = 
                  \coprod_{\substack{\mu_{ki},\mu_{kj}\in\calO(\calU) \\
                                     \lambda\in\ABST(\mu_{ki}) \\
                                      \xi\in\ABST(\mu_{kj})}
                          } P(\lambda,\xi).\]
Note that since both $\tlambda$ and $\txi$ are open embeddings, the space $P(\lambda,\xi)$
is homeomorphic to $\tlambda(\tU_i)\cap\txi(\tU_j)\subseteq \tU_k$ and hence carries
a canonical smooth structure.

In the general theory of internal categories of fractions, the space
$\mbox{ ComSq}\,(\mathcal{C}(\mathfrak{U}))$ representing commutative squares is obtained as an equalizer
of composition maps from the space of  all possible squares to  $\CU_1$.   In this case,
the space encoding all possible squares is a large coproduct of charts,
\[\coprod \tU_\ell,\]
taken over all combinations of $(\lambda,\gamma,\xi,\delta)\in\ABST(\mu_{ki})
\times\ABST(\mu_{i\ell})\times\ABST(\mu_{kj})\times\ABST(\mu_{j\ell})$. 
The two composition maps send a point $(x,\lambda,\gamma,\xi,\delta)$ to
$(x,\boldalpha_{ki\ell}(\lambda\otimes\gamma))$ and $(x,\boldalpha_{kj\ell}(\xi\otimes\delta))$
respectively.   So we see that the space of commutative squares is a coproduct of equalizers. 
Moreover, in our situation the two parallel embeddings $\boldalpha_{ki\ell}(\lambda\otimes\gamma)$
and $\boldalpha_{kj\ell}(\xi\otimes\delta)$ either agree everywhere or nowhere, and hence
the space becomes simply a coproduct of charts:
\[\mbox{ ComSq}\,(\mathcal{C}(\mathfrak{U}))=\coprod\tU_\ell\]
where the coproduct is taken over all $(\lambda,\gamma,\xi,\delta)\in
\ABST(\mu_{ki})\times\ABST(\mu_{i\ell})\times\ABST(\mu_{kj})\times\ABST(\mu_{j\ell}) $ such that
$\boldalpha_{ki\ell}(\lambda\otimes\gamma)=\boldalpha_{kj\ell}(\xi\otimes\delta)$.  

There is a projection map $\varphi\colon \mbox{ ComSq}\,(\mathcal{C}(\mathfrak{U}))\to \mbox{ Cospan}\,(\mathcal{C}(\mathfrak{U})))$
given by
\[\varphi(x,\lambda,\gamma,\xi,\delta) = (\lambda,\tgamma(x),\tdelta(x),\xi).\]
Following~\cite{PSi}, we show that  the internal version of the Ore condition is satisfied in 
$\CU$:

\begin{lma}[The Topological Ore Condition for $\CU$]
The map $\varphi\colon \mbox{ ComSq}\,(\mathcal{C}(\mathfrak{U}))\to \mbox{ Cospan}\,(\mathcal{C}(\mathfrak{U}))$
is a surjective local homeomorphism.
\end{lma}

\begin{proof}
All the $\tlambda$s
are embeddings and the map $\varphi$  is an embedding on each connected component, and so
$\varphi$ is  local homeomorphism.   Lemma~\ref{strong_compatibility} implies that $\varphi$
is surjective.  
\end{proof}

Next we check the  internal version of condition (WC).  This requires us to consider a
map $\mathfrak m$ from the space of diagrams in $\calC(\mathfrak{U})_1$ of the form

\begin{equation}\label{P-diagram}
\begin{tikzpicture}[xscale=1.5,yscale=-1.2]
    \def \z {35}
    \node (A) at (0, 0) {};
    \node (B) at (1, 0) {};
    \node (C) at (2, 0) {};
    \node (D) at (3, 0) {};
    \path (A) edge [->]node [auto] {$\scriptstyle{(x,\lambda_1)}$} (B);
    \path (C) edge [->]node [auto] {$\scriptstyle{(y,\lambda_4)}$} (D);
    \path (B.\z) edge [->]node [auto] {$\scriptstyle{(\tlambda_1(x),\lambda_2)}$} (C.180-\z);
    \path (B.360-\z) edge [->]node [auto,swap] {$\scriptstyle{(\tlambda_1(x),\lambda_3)}$} (C.180+\z);
\end{tikzpicture}
\end{equation}
with $\lambda_2\otimes\lambda_1=\lambda_3\otimes\lambda_1$, $\lambda_4\otimes\lambda_2=\lambda_4
\otimes\lambda_3$ and $\tlambda_2\tlambda_1(x)=y$, to the space of diagrams of the form 

\begin{equation}\label{fork}
\begin{tikzpicture}[xscale=1.5,yscale=-1.2]
    \def \z {35}
    \node (B) at (1, 0) {};
    \node (C) at (2, 0) {};
    \node (D) at (3, 0) {};
    \path (C) edge [->]node [auto] {$\scriptstyle{(y,\lambda_4)}$} (D);
    \path (B.\z) edge [->]node [auto] {$\scriptstyle{(z,\lambda_2)}$} (C.180-\z);
    \path (B.360-\z) edge [->]node [auto,swap] {$\scriptstyle{(z,\lambda_3)}$} (C.180+\z);
\end{tikzpicture}
\end{equation}
with $\lambda_4\otimes\lambda_2=\lambda_4\otimes\lambda_3$ and $\tlambda_2(z)=y=
\tlambda_3(z)$.  The map $\mathfrak m$ is the obvious projection.

\begin{lma}[The Topological  Weak Cancellability Condition for $\CU$]  The projection map
$\mathfrak{m}$ described above is a surjective local homeomorphism.  
\end{lma}

\begin{proof} 
The space encoding diagrams of the form \eqref{P-diagram} is a coproduct of charts,

\[\coprod \tU_i,\]
where $U_i\le U_j\le U_k\le U_\ell\mathrm{\, in \,}\calO(\calU) $ and $\lambda_1\in\ABST(\mu_{ji})$, 
$\lambda_2,\lambda_3\in\ABST(\mu_{kj})$, $\lambda_4\in\ABST(\mu_{\ell k})$ satisfying 
$\lambda_2\otimes\lambda_1=\lambda_3\otimes\lambda_1$ and $\lambda_4\otimes\lambda_2=
\lambda_4\otimes\lambda_3$.  Lemma~\ref{left-cancel}  (or~\ref{right-cancel}) then implies
that $\lambda_2=\lambda_3$.  So in fact this is 

\[\coprod_{\substack{U_i\le U_j\le U_k\le U_\ell\mathrm{ \, in \, }\calO(\calU) \\
\lambda_1\in\ABST(\mu_{ji}) \\
\lambda_2\in\ABST(\mu_{kj}) \\
\lambda_4\in\ABST(\mu_{\ell k})}}\tU_i,\]
and we denote its elements by $(x,\lambda_1,\lambda_2,\lambda_4)$.

Similarly, in the diagram \eqref{fork}, Lemma~\ref{left-cancel} implies that $\lambda_2=\lambda_3$.
Hence, this space is

\[\coprod_{\substack{U_j\le U_k\le U_\ell \mathrm{ \, in \,}\calO(\calU) \\
\lambda_2\in\ABST(\mu_{kj}) \\
\lambda_4\in\ABST(\mu_{\ell k})}}\tU_j,\]
and we denote its elements by $(z,\lambda_2,\lambda_4)$.
In this notation, the  projection map $\mathfrak m$ we are considering is defined by
\[(x,\lambda_1,\lambda_2,\lambda_4)\mapsto (\tlambda_1(x),\lambda_2,\lambda_4).\]
Since this is an embedding when restricted to any connected component, it is clear that this
is a local homeomorphism  (even a local diffeomorphism).   To see that it is surjective,
given $(z,\lambda_2,\lambda_4) $ where $z\in \tU_j$, $\lambda_2\in\ABST(\mu_{kj})$ and
$\lambda_4\in\ABST(\mu_{\ell k})$, we can define $U_i = U_j$, $x=z$, 
and $\lambda_1\in \ABST(\mu_{ii})$ such that $\tlambda_1$
is the identity on $\tU_i$; this gives an object $(x,\lambda_1,\lambda_2,\lambda_4)$ such that
$\mathfrak{m} (x,\lambda_1,\lambda_2,\lambda_4) = (z,\lambda_2,\lambda_4)$.
\end{proof}

Thus, \cite{PSi} says that  we have the conditions necessary to construct a
topological groupoid of fractions.
This construction gives us the following atlas groupoid $\GU$.
The space of objects remains the same:

\[\GU_0=\coprod_{\tU_i\in\,\calU}\tU_i.\]
To obtain the space of arrows $\GU_1$, we start with the space of spans
\[\mathrm{Span}(\CU)=\coprod_{\mu_{ji},\mu_{ki}\in\calO(\calU)}\ABST(\mu_{ji})\times
\tU_i\times\ABST(\mu_{ki}).\]
%
%
We define $\GU_1$ as a coequalizer encoding the diagrams of the form
\eqref{spansofspans}, which in this case is the quotient of the space $\mathrm{Span}(\CU)$
by the equivalence relation $\RU$, described as follows:  
Suppose  $U_i,U_{i'}\leq U_j$, and $U_i,U_{i'}\leq U_k$  in $\calO(\calU)$,
and fix $(\lambda,x,\nu), (\lambda',x',\nu')$ for $x\in\tU_i$, $x'\in\tU_{i'}$,
$\lambda\in\ABST(\mu_{ji})$, $\lambda'\in\ABST(\mu_{ji'})$,
$\nu\in\ABST(\mu_{ki})$, $\nu'\in\ABST(\mu_{ki'})$.   Then,

\[((\lambda,x,\nu), (\lambda',x',\nu'))\in\RU\]
if and only if there is a chart $\tU_\ell$ with a point $y\in\tU_\ell$ and elements
$\kappa\in\ABST(\mu_{i\ell})$ and $\kappa'\in\ABST(\mu_{i'\ell})$ such that

\[\tkappa(y)=x,\quad \tkappa'(y)=x',\quad \boldalpha_{ji\ell}(\lambda\otimes\kappa)=
\boldalpha_{ji'\ell}(\lambda'\otimes\kappa') \mbox{ and }
\boldalpha_{ki\ell}(\nu\otimes\kappa)=\boldalpha_{ki'\ell}(\nu'\otimes\kappa').\]
The space of arrows of the groupoid is then 
\[\GU_1=\mathrm{Span}(\CU)/\RU.\]
The structure maps of the groupoid  are defined as follows.   
For $x\in\tU_i$, $\lambda\in\ABST(\mu_{ji})$ and $\nu\in\ABST(\mu_{ki})$:

\begin{eqnarray*}
s[\lambda,x,\nu] & = & \tlambda(x) \\
t[\lambda,x,\nu] & = &\tnu(x) \\
            u(x) & = & [\boldalpha_i(e_i),x,\boldalpha_i(e_i)]
\end{eqnarray*}
where $e_i$ is the identity of $G_i$.   By Remark~\ref{D:rmks}(\ref{D:prescomp}),
the maps $s$ and $t$ are well defined on equivalence classes.

We also need to define the composition operation on the arrow space of this groupoid.
Suppose we have composable elements  $[\lambda_1,x_1,\nu_1]$ and $[\lambda_2,x_2,\nu_2]$ with
$\lambda_1\in\ABST(\mu_{j_1 i_1})$, $\nu_1\in\ABST(\mu_{j_0 i_1})$, 
$\lambda_2\in\ABST(\mu_{j_0 i_2})$, $\nu_2\in\ABST(\mu_{j_2 i_2})$;  so we know that 
$\tnu_1(x_1)=\tlambda_2(x_2)$.  By Lemma~\ref{strong_compatibility} there is a chart
$\tU_k$ with a point $y\in\tU_k$ and abstract embeddings $\kappa_n\in\ABST(\mu_{i_nk})$ for $n=1,2$
such that $\tkappa_n(y)=x_n$ for $n=1,2$ and $\boldalpha_{j_0 i_1 k}(\nu_1\otimes\kappa_1)=
\boldalpha_{j_0 i_2 k}(\lambda_2\otimes\kappa_2)$.   Then we define

\[m([\lambda_2,x_2,\nu_2],[\lambda_1,x_1,\nu_1])=
[\boldalpha_{j_1 i_1 k}(\lambda_1\otimes\kappa_1),y,\boldalpha_{j_2 i_2 k}(\nu_2\otimes\kappa_2)].\]
The fact that this is well defined on equivalence classes was proved in~\cite{PSi}.

Lastly, the inverse of an arrow in $\GU$ is given by
\[i[\lambda,x,\nu]=[\nu,x,\lambda].\]
This is the topological groupoid produced by the construction of~\cite{PSi}. 

We now want to show that this groupoid  represents an orbifold; that is, it has the required
topological properties.    Specifically, we show that the arrow space is Hausdorff, that
there is a natural smooth structure such that all structure maps are local diffeomorphisms, 
and that the diagonal $(s,t)\colon\GU_1\to\GU_0\times\GU_0$ is a proper map.  

\begin{prop}
The space of arrows $\GU_1$ is Hausdorff.
\end{prop}

\begin{proof}
We need to show that the quotient $\GU_1=\mathrm{Span}(\CU)/\RU$ is Hausdorff.
The original space of arrows $\CU_1$, defined by a disjoint union of charts, is clearly
Hausdorff, as is  $\mbox{\rm Span}(\CU)$.   Therefore it suffices to show that the subspace

\[\RU=\{((\lambda,x,\nu),(\lambda',x',\nu'))\}\]
defined by the equivalence relation above is closed in
$\mbox{\rm Span}(\CU)\times\mbox{\rm Span}(\CU)$.

We show that the complement is open.  So suppose that $((\lambda,x,\nu) (\lambda',x',\nu'))
\in (\ABST(\mu_{ji})\times\tU_i\times\ABST(\mu_{ki}))\times(\ABST(\mu_{j'i'})
\times\tU_{i'}\times\ABST(\mu_{k'i'}))\subset
(\mathrm{Span}(\CU)\times\mathrm{Span}(\CU))\smallsetminus\RU$, so that
$(\lambda,x,\nu)$ and $ (\lambda',x',\nu')$ are not related under $\RU$.
We produce an open set $\cal{O} $ containing $((\lambda,x,\nu),(\lambda',x',\nu'))$
such that $\RU\cap\cal{O} =\varnothing$.

If $j'\neq j$ or $k'\neq k$, then we can define  $\cal{O} = 
\left(\{\lambda\}\times \tU_i\times\{\nu\}\right)
\times\left(\{\lambda'\}\times \tU_{i'}\times\{\nu'\}\right)$. 

If $j=j'$ and $k=k'$ but $\tlambda(x)\neq \tlambda'(x')$, we can use the Hausdorff property
of $\tU_j$ to find disjoint open neighbourhoods $\tU_{\tlambda(x)}$ and $\tU_{\tlambda'(x')}$,
and take their preimages under $\tlambda$ and $\tlambda'$ to obtain open
neighbourhoods $\tU_x$ of $x$ and $\tU_{x'}$ of $x'$.   These give us open neighbourhoods
$\{\lambda\}\times \tU_x\times\{\nu\}\subseteq \{\lambda\}\times \tU_i\times\{\nu\}$
and $\{\lambda'\}\times \tU_{x'}\times\{\nu'\}\subseteq \{\lambda'\}\times \tU_{i'}\times\{\nu'\}$
of $(\lambda,x,\nu)$ and $(\lambda',x',\nu')$ respectively in $\mbox{\rm Span}(\CU)$,  and we define
$\cal{O} = \left(\{\lambda\}\times\tU_x\times\{\nu\}\right)\times\left(\{\lambda'\}\times
\tU_{x'}\times\{\nu'\}\right)$.  A similar argument produces $\cal{O}$ when $\tnu(x)\neq \tnu'(x')$. 

Lastly, we consider the case where  $j=j'$, $k=k'$, $\tlambda(x)=\tlambda'(x')$
and $\tnu(x)=\tnu'(x')$, but $(\lambda,x,\nu)$ and $(\lambda',x',\nu')$ are not related
under $\RU$.   By Lemma~\ref{strong_compatibility} there is a chart $\tU_\ell$ with a point
$y$ and a pair of abstract embeddings $\kappa\in\ABST(\mu_{i\ell})$ and
$\kappa'\in\ABST(\mu_{i'\ell})$, such that $\tkappa(y)=x$, $\tkappa'(y)=x'$ and
$\boldalpha_{ji\ell}(\lambda\otimes\kappa)=\boldalpha_{ji'\ell}(\lambda'\otimes\kappa')$.
Since $(\lambda,x,\nu)$ and $(\lambda',x',\nu')$ are not related by $\RU$, we know that
\begin{equation}\label{eq-02}
\boldalpha_{ki\ell}(\nu\otimes\kappa)\neq\boldalpha_{ki'\ell}(\nu'\otimes\kappa').
\end{equation}
Now we consider the open neighbourhoods $\{\lambda\}\times\tkappa(\tU_\ell)\times\{\nu\}$ of
$(\lambda,x,\nu)$ and $\{\lambda'\}\times\tkappa'(\tU_\ell)\times\{\nu'\}$ of
$(\lambda',x',\nu')$, and define
\[\cal{O}= \left(\{\lambda\}\times \tkappa(\tU_\ell)\times\{\nu\}\right)
\times\left(\{\lambda'\}\times \tkappa'(\tU_\ell)\times\{\nu'\}\right).\]
We claim that $\RU\cap \cal{O} =\varnothing$:  suppose by contradiction that there are points
$z\in\tkappa(\tU_\ell)$ and $z'\in\tkappa'(\tU_\ell)$ such that $(\lambda,z,\nu)
\sim(\lambda',z',\nu')$.   Then there is a chart $\tU_m$ with a point $u\in\tU_m$ 
and abstract embeddings $\theta\in\ABST(\mu_{im})$ and $\theta'\in\ABST(\mu_{i'm})$
such that $\ttheta(u)=z$, $\ttheta'(u)=z'$, $\boldalpha_{jim}(\lambda\otimes\theta)=
\boldalpha_{ji'm}(\lambda'\otimes\theta')$ and $\boldalpha_{kim}(\nu\otimes\theta)=
\boldalpha_{ki'm}(\nu'\otimes\theta')$.
Without loss of generality, we may assume that $U_m\le U_\ell$ in $\calO(\calU)$.
Since $\ttheta(u)=z\in\tkappa(\tU_\ell)$, then by Lemma~\ref{interpolation}
there is a unique $\zeta\in\ABST(\mu_{\ell m})$ such that
$\boldalpha_{i\ell m}(\kappa\otimes\zeta)=\theta$.   Now, 
\begin{eqnarray*}
\boldalpha_{ji'm}(\lambda'\otimes\boldalpha_{i'\ell m}(\kappa'\otimes\zeta)) & =
& \boldalpha_{j\ell m}(\boldalpha_{ji'\ell}(\lambda'\otimes\kappa')\otimes\zeta) \\
& = &\boldalpha_{j\ell m}(\boldalpha_{ji\ell}(\lambda\otimes\kappa)\otimes\zeta) \\
& = & \boldalpha_{jim}(\lambda\otimes\boldalpha_{i\ell m}(\kappa\otimes\zeta)) \\
& = &\boldalpha_{jim}(\lambda\otimes\theta)\\ &=& \boldalpha_{ji'm}(\lambda'\otimes\theta').
\end{eqnarray*}
Since $\boldalpha_{ji'm}$ is an isomorphism, using the previous identity and
Lemma~\ref{left-cancel}, we get that $\boldalpha_{i'\ell m}(\kappa'\otimes\zeta)=\theta'$.
Hence,
\begin{eqnarray*}
\boldalpha_{k\ell m}(\boldalpha_{ki\ell}(\nu\otimes \kappa)\otimes\zeta) & = &
\boldalpha_{kim}(\nu\otimes\boldalpha_{i\ell m}(\kappa\otimes\zeta)) \\
& = & \boldalpha_{kim}(\nu\otimes\theta) \\
& = & \boldalpha_{ki'm}(\nu'\otimes\theta') \\
& = & \boldalpha_{ki'm}(\nu'\otimes\boldalpha_{i'\ell m}(\kappa'\otimes\zeta)) \\
& = & \boldalpha_{k\ell m}(\boldalpha_{ki'\ell}(\nu'\otimes\kappa')\otimes\zeta)
\end{eqnarray*}
By Lemma~\ref{right-cancel}, this implies that $\boldalpha_{ki\ell}(\nu\otimes\kappa)=
\boldalpha_{ki'\ell}(\nu'\otimes\kappa')$, contradicting \eqref{eq-02}.
This completes the proof of the claim and the proof of the proposition.
\end{proof}

\begin{lma}
The source map $s$ of the groupoid $\GU_1$ is a local homeomorphism.
\end{lma}

\begin{proof}
By Lemma~\ref{left-cancel}, two points $(\lambda,x,\nu)$ and $(\lambda,x',\nu)$
in $\mathrm{Span}\,(\CU)$ with $x,x'\in\tU$, $\lambda\in\ABST(\mu_{VU})$ and
$\nu\in\ABST(\mu_{WU})$ are identified in $\GU_1$ if and only if $x=x'$.
Hence the points of the form $[\lambda,y,\nu]$ in $\GU_1$ (for $y\in\tU$) form an open neighbourhood
of $[\lambda,x,\nu]$ which is homeomorphic to $\tU$, and $s$ is a homeomorphism onto its image
when restricted to this neighbourhood.
\end{proof}

Note that the space $\mathrm{Span}\,(\CU)$ is a manifold and the quotient map
$\mathrm{Span}\,(\CU)\to\GU_1$ is a local homeomorphism.   In fact, it is a homeomorphism/embedding
when restricted to each connected component, so $\GU_1$ inherits the structure of
being a smooth manifold and the structure maps are local diffeomorphisms with
respect to this structure.

\begin{prop}
The diagonal $(s,t)\colon\GU_1\to\GU_0\times\GU_0$ is proper.
\end{prop}

\begin{proof}
We begin by showing that the fibers of $(s,t)$ are finite, hence compact. 
Since $\GU$ is a groupoid, it is sufficient to consider the fibers of the form $(s,t)^{-1}(x,x)$.
So let $x\in\GU_0$, thus $x\in\tU_i$ for some $i$, and consider $(s,t)^{-1}(x,x)$.
First consider the elements in this fiber of the form $[\lambda,x',\lambda']$ with 
$\lambda,\lambda'\in\ABST(\mu_{ii})$ (which is isomorphic to $\mathsf{G}_i$ via
$\boldalpha_i^{-1}$) and $x'\in \tU_i$, such that $\tlambda(x')=x=\tlambda'(x')$.
For any $g\in G_i$, we know that
$[\lambda,x',\lambda']=[\lambda\cdot g,(\rho_i(g^{-1}))(x'),\lambda'\cdot g]$.
Since the right action of $G_i$ is transitive on $\ABST(\mu_{ii})$,  each $[\lambda,x',\lambda']$
has a representation for which  $\lambda=\boldalpha_i(e)$, and consequently
$\tlambda=\mbox{id}_{\tU_i}$.   So in this representation, we have  $[\lambda,x,\lambda']$,
with $\lambda=\boldalpha_i(e)$  and $\lambda'$ such that $\tlambda'(x)=x$.

Fix one of these $\lambda'$. We claim that all the other arrows in $(s,t)^{-1}(x,x)$
are of the form $[\lambda,x,g\cdot\lambda']$ with $g\in (G_i)_x=
\{g\in G_i\textrm{ such that }\rho_i(g_i)\cdot x=x\}$.   To see that this is true, suppose that
$[\nu,y,\nu']\in(s,t)^{-1}(x,x)$ with $y\in \tU_j$ and $\nu,\nu'\in \ABST(\mu_{ij})$.
If $g\in G_i$ is the unique element such that $g\cdot\boldalpha_{iij}(\lambda'\otimes\nu)=\nu'$, 
then $[\lambda,x,g\cdot\lambda']=[\nu,y,\nu']$.   This implies that  $(s,t)^{-1}(x,x)$ is finite.  

It remains to show that $(s,t)$ is a closed map.
We begin by showing that the image $(s,t)(\GU_1)\subseteq \GU_0\times\GU_0$ is closed.
Let $(x, y) \in\GU_0\times\GU_0$, with $x\in\tU_j$ and $y\in\tU_k$.
If $(x,y)\in(\GU_0\times\GU_0) \smallsetminus (s,t)(\GU_1)$, then the images $\bar{x}$ and
$\bar{y}$ in the quotient space are distinct.  
Since the quotient space is Hausdorff, there are disjoint
open neighbourhoods $U_{\bar{x}}$ and $U_{\bar{y}}$ of $\bar{x}$ and $\bar{y}$ in $\GU_0 / \GU_1$.  
Define $\tU_{\bar{x}}$ to  be the preimage of $U_{\bar{x}}$ in $\tU_j$ under the quotient map,
and similarly let $\tU_{\bar{y}}$ be the preimage of $U_{\bar{y}}$ in $\tU_k$.
Then $(\tU_{\bar{x}}\times\tU_{\bar{y}})\cap (s,t)(\GU_1)=\varnothing$, and we have shown
that the complement of  $(s,t)(\GU_1)$ is open.  

Now let $C\subseteq \GU_1$ be a closed subset, and let $(x,y)\in (\GU_0\times\GU_0)
\smallsetminus (s,t)(C)$, where $x\in\tU_j$ and $y\in\tU_k$.
Without loss of generality we may assume that $(s,t)^{-1}(x,y)
\neq \varnothing$. We  have shown earlier that it is finite, and therefore we have
$(s,t)^{-1}(x,y)=\{[\lambda_1,z_1,\lambda'_1],\ldots,[\lambda_n,z_n,\lambda'_n]\}$
with $z_i\in \tU_{m_i}$ and $\lambda_i\in\ABST(\mu_{jm_i})$, $\lambda'_i\in\ABST(\mu_{km_i})$
for each $i=1,\cdots,n$.
Since $C\subseteq \GU_1$ is closed, for each $i$ there is an open neighbourhood
$\tU'_{m_i}$ of $z_i$ in $\tU_{m_i}$, such that the set
$V_i=\{\lambda_i\}\times \tU'_{m_i}\times\{\lambda_i'\}$ does not intersect $C$.

Both $s$ and $t$ are homeomorphisms when restricted to each of the $V_i$, so $s(V_i)\times t(V_i)$
is open in $\GU_0\times\GU_0$.
Then $W_x\times W_y:=\left(\bigcap_{i=1}^n s(V_i)\right)\times\left(\bigcap_{i=1}^n t(V_i)\right)$
is an open neighbourhood of the point $(x,y)$.
However, $(s,t)^{-1}(W_x\times W_y)$ may contain connected components that are different
from the $V_i$ chosen before and hence may have a nonempty intersection with $C$.
In order to make sure that there are no such components we need to choose smaller connected
neighbourhoods $W'_x\subseteq W_x\subseteq \tU_j$ and $W'_y\subseteq W_y\subseteq \tU_k$
such that the isotropy groups of all points in $W'_x$ are subgroups of the isotropy group
of $x$ and similarly for $W'_y$.  We show how to construct $W'_x$.
($W'_y$ can be found similarly.)  For each $g\in G_j$, consider $\rho_j(g)(W_x)$.
If $x\in \rho_j(g)(W_x)$, let $W_x^g=\rho_j(g)(W_x)$.
If $x\not\in \rho_j(g)(W_x)$, note that $x$ is also not in the closure
$\overline{\rho_j(g)(W_x)}$ and let $W_x^g=W_x\smallsetminus\overline{\rho_j(g)(W_x)}$
Now define $W_x'=\bigcap_{g\in G_j}W_x^g$.
Then $(s,t)^{-1}(W_x'\times W_y')\subseteq \bigcup V_i$ and hence $W_x'\times W_y'$
has an empty intersection with $(s,t)(C)$.
\end{proof}

\section{Constructing an Orbifold Atlas from a Groupoid}\label{S:atlastogroupoid}

Given a smooth \'etale groupoid $\calG$ with a proper diagonal, we will construct an orbifold
atlas for its quotient space $\calG_0/\calG_1$.
We begin by reviewing a couple of results about these groupoids from~\cite{MP}. 

Let $s, t\colon \calG_1\to\calG_0$ denote the source and target maps of this groupoid.
Since $s$ and $t$ are \'etale, the preimage $s^{-1}(x)\cap t^{-1}(x)=(s,t)^{-1}(x,x)$ is a
discrete group for any point $x\in \calG_0$; and since
$(s,t)\colon\calG_1\to\calG_0\times\calG_0$ is proper, this pre-image is a finite group.
We denote it by $\calG_x$.

Each point $g\in \calG_1$ has a neighbourhood $W_g$ such that both $s$ and $t$ restrict
to homeomorphisms on $W_g$.  Such neighbourhoods are called {\em local bisections}.
When $s\colon W_g\to \tU$ and $t\colon W_g\to \tV$, we also denote this local bisection
by its section $\sigma\colon \tU\to W_g\subseteq \calG_1$; i.e., $s\circ\sigma=\mbox{id}_{\tU}$.
In this case we say that this local bisection carries $\tU$ onto $(t\circ\sigma)(\tU)=\tV$.    
We will also consider the situation where we take $\tilde{V}$ to strictly contain the image $t(W_g)$; i.e., when $(t\circ\sigma)(\tU)\subseteq\tV$.  In this case we will 
say that the local bisection takes or embeds $\tU$ into $\tV$.

The properness of the diagonal implies
the existence of a collection of special open neigbourhoods for points $x\in\calG_0$, the space of objects.  In developing these, we will consider invariant sets $\tU $ of $\calG_0$, and we will denote its quotient $\tU/\calG_1$ in $\calG_0/\calG_1$ by $U$. 

\begin{dfn}
A connected simply connected open subset $\tU\subseteq\calG_0$ is a {\em translation open set}
if the group $G_{U}$ of local bisections $\sigma\colon \tU\to\calG_1$  (i.e., $s\circ\sigma=\mbox{id}_{\tU}$) which are defined on $\tU$  and carry 
$\tU$ onto itself (i.e., $(t\circ\sigma)(\tU)=\tU$) is finite and the homomorphism of Lie groupoids 
$G_{U}\ltimes \tU\to\calG|_{\tU}$, defined by evaluation $(\sigma,u)\mapsto \sigma(u)$, is an isomorphism.
\end{dfn}
 
These translation neighbourhoods form a basis for the topology on $\calG_0$, as do the local bisections
for the topology on $\calG_1$.  (This was shown for translation neighbourhoods in the proof of the part $4\Rightarrow 1$ of the main theorem of~\cite{MP}.  
Although that paper is about effective orbifolds, this part of the proof does
not use effectiveness.)

Translation neighbourhoods and local bisections will form the essential building blocks for 
the orbifold atlas.   Before defining the atlas, we prove some basic results
about the structure of these neighbourhoods.

\subsection{Translation Neighbourhoods and Local Bisections} \label{S:tr}

We start with a couple of results about the relationship between translation neighbourhoods
and local bisections.   For any two open subsets $\tU,\tV\subset\calG_0$, let $\mathsf{A}_{VU}$ 
denote the set of all local bisections $\sigma\colon \tU\to\calG(U,V)=s^{-1}(\tU)\cap t^{-1}(\tV)$
which are defined on $\tU$ and carry $\tU$ into $\tV$ in the sense that $(t\circ\sigma)(\tU)\subseteq \tV$.
Note that there is a canonical action of $G_{V}$ on $\mathsf{A}_{VU}$ from the left, and canonical action of $G_{U}$ on 
$\mathsf{A}_{VU}$ from the right. Both actions arise from the standard composition of bisections 
$(\sigma'\sigma)(u):=\sigma'(t\sigma(u))\sigma(u)$, where the composition on the righthand side is in $\calG$.
With these actions,  $\mathsf{A}_{UU }=G_U $ for any translation open set $\tU $.

\begin{lma}\label{D:point-open}
For each translation open set $\tV\subseteq \calG_0$ and  each $x\in\calG_0$ such that the orbit
$\calG_x$ intersects $\tV$, there is a translation open neighbourhood $\tU $ of $x$ 
such that the evaluation map $\mathsf{A}_{VU }\times \tU \to\calG(U, V)$ is a diffeomorphism.
\end{lma}

\begin{proof}
Since $\tV$ is a translation neighbourhood, $s^{-1}(x)\cap t^{-1}(\tV)$ is finite
(the intersection of any orbit with a translation neighbourhood is finite and 
each point has a finite isotropy group.)
For each $g\in s^{-1}(x)\cap t^{-1}(\tV)$, choose a local bisection $\tV_g$ containing $g$ and such that
$t(\tV_g)\subseteq \tV$.   Let $W_x=\bigcap s(\tV_g)$ and let $\tU $
be a translation neighbourhood of $x$ such that  $\tU \subseteq{W}_x$.
Then the evaluation map clearly defines an embedding $\mathsf{A}_{VU }\times \tU \to \calG(U, V)$.
This embedding is surjective by the way $\tU $ was constructed.
\end{proof}

Next we consider the relationship between two translation neighbourhoods $\tU $ and $\tV$
of $\calG_0$, where one is a subset of the other up to equivalence:   that is, such that
$\tU / \calG_1 \subseteq \tV/\calG_1$ in the quotient groupoid $\calG_0/\calG_1$.      
\begin{lma}\label{D:subcharts}
Let $\tU ,\tV\subseteq\calG_0$ be translation open subsets with $\tU $ connected and simply connected such that 
$\tU / \calG_1 \subseteq \tV/\calG_1$. Then the evaluation map $\mathsf{A}_{VU }\times \tU \to\calG(U, V)$ is a diffeomorphism and the canonical left $G_{V}$ action on $\mathsf{A}_{VU }$ is free and transitive.
\end{lma}

\begin{proof} 
Since $\tU / \calG_1 \subseteq \tV/\calG_1$, the map $s\colon s^{-1}(\tU )\cap t^{-1}(\tV)\to \tU $ is
surjective.  The previous lemma implies that it is a covering with $G_{V}$ as group of deck
transformations.   Since $\tU $ is connected and simply connected, this implies that
$s^{-1}(\tU )\cap t^{-1}(\tV)\cong G_{V}\times \tU $,
a disjoint union of local bisections which are all diffeomorphic to $\tU $. It follows also immediately that the action of $G_{V}$ on $\mathsf{A}_{VU }$ is free and transitive.
\end{proof}

By interchanging the roles of $s$ and $t$ in the previous proof, we obtain the following result.

\begin{cor}\label{D:subcharts2}
Let $\tU $ be a translation neighbourhood which is connected and simply connected and let $\tV$
be a translation neighbourhood such that  $\tU / \calG_1 \subseteq \tV/\calG_1$.  
Then $s^{-1}(\tV)\cap t^{-1}(\tU )$ consists of a disjoint union of local bisections, each of which is
diffeomorphic to $\tU $ via $t$. 
\end{cor}

\begin{rmks}\label{bis-rmks} The structure maps of the groupoid allow us to invert surjective local bisections, or invert them on their image when they are not surjective,  and compose them when their image and domain of definition match.  \end{rmks}

\subsection{The Atlas Construction}\label{S:atlascon}
We are now ready to construct an orbifold atlas $\mathfrak{U}$ from each
smooth \'etale groupoid $\calG$ with proper diagonal.

The first thing we need is a collection of orbifold charts as in condition (1) of
Definition~\ref{defin-02}.  We choose a collection $\calU$ of connected and simply connected
translation neighbourhoods  that cover the space of objects $\calG_0$.   These translation neighbourhoods will be denoted  by $ \tU_i$, and their quotients $\tU_i/\calG = U_i$.   Furthermore, we require that
whenever a point $\bar{x}\in \calG_0/\calG_1$ lies in the quotients $U_1$ 
and $U_2$ of two translation neighbourhoods ($\tU_1$ and $\tU_2$ respectively), 
there is a third translation neighbourhood $\tV$ with $\bar{x}\in V$ and  
$V\subseteq {U}_1\cap {U}_2$;  it is possible to do this because the translation
neighbourhoods form a basis for the topology of $\calG_0$ and the spaces are paracompact and Hausdorff. 

For each translation neighbourhood  $\tU$,  we get an {\em atlas chart}
\[(\tU,G_U,\rho_U,\pi_U) \]
where $G_U$ is the set of connected components of $s^{-1}(\tU)\cap t^{-1}(\tU )$, $\rho_U\colon G_U\times\tU\to\tU$ is the action map defined by evaluation and $\pi_U\colon \tU\to\tU/\calG_1 = U$ is the quotient map.

With these charts, the concrete embeddings are given as follows.  For charts $\tU_i$ and $\tU_j$
with $U_i\subseteq U_j$ in the quotient space,
$\CONCR(\mu_{ji})$ is the collection of embeddings obtained by taking composites
\begin{tikzpicture}[baseline, xscale=1.2,yscale=-1.2]
    \node (A0_0) at (0, -0.1) {$\tU_i$};
    \node (A0_1) at (1, -0.1) {$W$};
    \node (A0_2) at (2, -0.1) {$\tU_j$};
    \path (A0_0) edge [->]node [auto] {$\scriptstyle{s^{-1}}$} (A0_1);
    \path (A0_1) edge [right hook->]node [auto] {$\scriptstyle{t}$} (A0_2);
\end{tikzpicture}
for each connected component $W$ of $s^{-1}(\tU_i)\cap t^{-1}(\tU_j)$.
These are effective embeddings since each $W$ is a local bisection.  Note that different
local bisections $W$ may give rise to the same effective embedding in $\CONCR(\mu_{ji})$.

Next, we define the bimodules of embeddings $\ABST$ as in condition (2)
of Definition~\ref{defin-02}.
If  $\tU$ and $\tV$ are translation open subsets with $U\subseteq  V$, define
\[\ABST(\mu_{VU})={\mathsf A}_{VU},\]
with the actions by $G_{\tU}$ and $G_{\tV}$ as defined above.
These actions are free since $\calG$ is a groupoid, and the action by $G_{\tV}$ is transitive as shown in Lemma \ref{D:subcharts}.

For translation open sets $\tU_i, \tU_j, \tU_k$, with $U_i\subseteq U_j\subseteq U_k$,
the bimodule transformation
$\boldalpha_{kji}\colon\mathsf{A}_{\tU_k \tU_j}\otimes_{G_{U_j}}\mathsf{A}_{\tU_j \tU_i}\to\mathsf{A}_{\tU_k \tU_i}$ 
is induced by composition of local bisections: $(\sigma',\sigma)\mapsto\sigma'\sigma$
as noted in Remark  \ref{bis-rmks}.
Also, for any translation open set $\tU$, the transformation $\alpha_U\colon G_{\tU}\to\mathsf{A}_{\tU \tU}$ is the identity,
because of the way we have defined the group.
The usual associativity and unit conditions for an internal category imply associativity
and unit coherence for the $\boldalpha$s.

We define the oplax transformation $\boldrho$ from $\ABST$ to $\CONCR$ as in condition (3)
of Definition~\ref{defin-02} as follows: 
  \[\rho_{ji}:\boldrho_j\otimes_{G_j}{\ABST}(\mu_{ji})\rightarrow
  \CONCR(\mu_{ji})\otimes_{G_i^{\red}} \boldrho_i,  \, \, \rho_j(g_j)\otimes_{G_j} \sigma \to t\circ (g_j\sigma) \otimes_{G^{\red}_i} e_i^{\red}\]
for all $g_j \in G_j$ and $\sigma \in \ABST(\mu_{ji})$, where $g_j \sigma$ is constructed as in  Section \ref{S:tr}.  A straightforward computation shows that this is well-defined (i.e. it does not depend on the representative chosen for $\rho_j(g_j)\otimes_{G_j} \sigma $ in $\boldrho_j\otimes_{G_j}{\ABST}(\mu_{ji})$, that $\rho_{ji}$ commutes with the actions of $G_{U_i}$ and $G_{U_j}$, and that it is surjective.

Finally, we need to check condition  (\ref{D:transitivity}): let us suppose that $\sigma$ and $\tau$ are two objects of 
$\ABST(\mu_{ji})$, such that $\rho_{ji}(e_j^{\operatorname{red}}\otimes\sigma)=
\rho_{ji}(e_j^{\operatorname{red}}\otimes\tau)$. This implies that
the following identity holds in $\CONCR(\mu_{ji})\otimes_{G_i^{\red}} \boldrho_i$:

\[t\circ\sigma\otimes_{G_i^{\operatorname{red}}}e_i^{\operatorname{red}} = 
t\circ\tau\otimes_{G_i^{\operatorname{red}}}e_i^{\operatorname{red}},\]
hence we have $t\circ\sigma=t\circ\tau$. So we can compose (as bisections) $\tau$
with the inverse $\sigma^*$ of $\sigma$; we denote by
$g$ the composition $\sigma^*\tau\in\mathsf{A}_{\widetilde{U}_i,\widetilde{U}_i}=G_{\widetilde{U}_i}$.
Then we have $\sigma g=\tau$, so condition (3b) is satisfied.\\

\begin{rmk}
The atlases described in Example~\ref{E:z3p2} are constructed via this process from
the groupoids described in Example~\ref{E:z3}
by taking the  $U_i$ as translation neighbourhoods.  
\end{rmk}

\section{Equivalences}\label{sec-7}
In this section we will define a notion of equivalence for orbifold atlases.   This definition generalizes
the notion of atlas equivalence for effective orbifolds,  and the constructions described
in Sections 5 and 6 give us a one-to-one correspondence between equivalence classes
of orbifold atlases and Morita equivalence classes of orbifold groupoids.

\subsection{Equivalence of Orbifold Atlases}

Recall that a Satake atlas $\calU$ refines another Satake atlas $\calV$ precisely when
the covering by open subsets induced by $\calU$ refines the covering by open subsets
induced by $\calV$,  and there are Satake embeddings from the charts in $\calU$ to the
charts in $\calV$.   With this definition, two orbifold atlases are equivalent if they
have a common refinement.   For Satake atlases, the fact that $\calU$ is a refinement of
$\calV$ implies that $\calU\cup\calV$ can be made into a larger Satake atlas by adding
sufficiently many smaller charts to satisfy the local compatibility condition.  
This addition of smaller charts can be done in a canonical way since any connected
component of an invariant subset of a chart inherits a chart structure.

For ineffective orbifolds we need to adjust the definition of refinement so that it will
still imply that the union of an atlas with a refinement has a canonical atlas structure.  
In addition to requiring a refinement of coverings as in the effective case, we will also
ask for some compatibility information between charts in the  atlas and charts of its refinements.  
Thus, we will require the existence of certain atlas bimodules connecting charts of the
refinement to those of the atlas, and these need to be suitably compatible with the bimodules
that make up the two atlas structures.   For readers interested in the origin of the
compatibility conditions, the formal framework is as follows (if one just needs
the definition, the concrete data and conditions needed are spelled out
in Definition~\ref{refinement} below):

Let $\mathfrak{U}$ be an atlas on $X$ determined by a pair of pseudofunctors and a
pseudonatural transformation

\[
\begin{tikzpicture}[xscale=1.8,yscale=-1.2]
    \node (A0_0) at (0, 0) {$\calO(\calU)$};
    \node (A0_2) at (2, 0) {$\Groupprof$,};
    \node (A0_1) at (1, 0) {$\Downarrow\,\boldrho_{\mathfrak{U}}$};
    \path (0.3, -0.2) edge [->,bend right=18]node [auto] {$\scriptstyle{\ABST_{\mathfrak{U}}}$} (A0_2);
    \path (0.3, 0.2) edge [->,bend left=18]node [auto,swap] {$\scriptstyle{\CONCR_{\mathfrak{U}}}$} (A0_2);
\end{tikzpicture}
\]
and let $\mathfrak{V}$ be an atlas (on the same space $X$) determined by a pair of pseudofunctors
and a pseudonatural transformation

\[
\begin{tikzpicture}[xscale=1.8,yscale=-1.2]
    \node (A0_0) at (0, 0) {$\calO(\calV)$};
    \node (A0_2) at (2, 0) {$\Groupprof$.};
    \node (A0_1) at (1, 0) {$\Downarrow\,\boldrho_{\mathfrak{V}}$};
    \path (0.3, -0.2) edge [->,bend right=18]node [auto] {$\scriptstyle{\ABST_{\mathfrak{V}}}$} (A0_2);
    \path (0.3, 0.2) edge [->,bend left=18]node [auto,swap] {$\scriptstyle{\CONCR_{\mathfrak{V}}}$} (A0_2);
\end{tikzpicture}
\]
Then a \emph{refinement module} from $\mathfrak{U}$ to $\mathfrak{V}$ is given by the bimodule

\[
\mathsf{R}\colon \calO(\calU)\longarrownot\longrightarrow\calO(\calV)\]
defined by

\[\mathsf{R}(V_j,U_i)=\left\{\begin{array}{ll}
                           \{*\}     & \mbox{ if } U_i\subseteq V_j;\\
                           \varnothing & \mbox{ else},\end{array}\right.
\]
together with the structure to make this both a  module over $\Groupprof$ from $\ABST_{\mathfrak{U}}$
to $\ABST_{\mathfrak{V}}$,  and also a module over $\Groupprof$ from $\CONCR_{\mathfrak{U}}$ to 
$\CONCR_{\mathfrak{V}}$.   Equivalently, we need the data to construct a pseudofunctor  
on the bipartite category $\mbox{Bipart}\,(\mathsf{R})$ (also called the \emph{cograph} of
$\mathsf R$ \cite{St}),
which restricts to $\ABST_{\mathfrak{U}}$ on $\calO(\calU)$ and to $\ABST_{\mathfrak{V}}$ on
$\calO(\calV)$, and similarly the data to construct a pseudofunctor on this same bipartite category that
restricts to $\CONCR_{\mathfrak{U}}$ on $\calO(\calU)$ and to $\CONCR_{\mathfrak{V}}$ on
$\calO(\calV)$.

\begin{dfn}\label{refinement}
Given two orbifold atlases $\mathfrak{U}=(\calU,\ABST_{\mathfrak{U}},\boldrho_{\mathfrak{U}})$
and $\mathfrak{V}=(\calV,\ABST_{\mathfrak{V}},$ $\boldrho_{\mathfrak{V}})$ for the same topological
space $X$, we say that $\mathfrak{U}$ is a \emph{refinement} of $\mathfrak{V}$ if the following
conditions are satisfied:

\begin{enumerate}
 \item The Satake atlas $\cal{U}^{\red}$ is a refinement of the Satake atlas $\cal{V}^{\red}$.

 \item For any chart $\tV$ in $\calV$ and chart $\tU$ in $\calU$ with a point $x\in U\cap V
  \subseteq X$, there is a chart $\tW$ in $\calU$ with $x\in W\subseteq U\cap V$.

 \item\label{D:modules}  Whenever $\tU_i\in\calU$, $\tV_j\in \calV$ and $U_i\subseteq V_j$
   in $\calO(X)$, there is an atlas bimodule  $\mathsf{A}_{ji}=\mathsf{A}(V_j, U_i)\colon G_i
   \longarrownot\longrightarrow G_j$ with a surjective 2-cell $\rho_{ji}$ as in the following diagram:

   \[
   \begin{tikzpicture}[xscale=1.2,yscale=-0.8]
    \node (A0_0) at (0, 0) {$G_i$};
    \node (A0_2) at (2, 0) {$G_j$};
    \node (A1_1) at (1.15, 1) {$\rho_{ji}$};
    \node (A2_0) at (0, 2) {$G_i^{\red}$};
    \node (A2_2) at (2, 2) {$G_j^{\red}$};
    \node (A1_0) [rotate=225] at (0.7, 1) {$\Rightarrow$};

    \path (A0_0) edge [->]node [auto,swap] {$\scriptstyle{\boldrho_i}\,$} (A2_0);
       \slashing{\path (A0_0) -- node [rotate=90] {$\scriptscriptstyle{\slash}$} (A2_0);}{}
    \slashing
       {\path (A2_0.350) -- node [auto,swap] {$\scriptstyle{\mathsf{C}(V_j, U_i)}$} (A2_2.190);
       \path (A2_0.350) edge [->] node [] {$\scriptscriptstyle{\slash}$} (A2_2.190);}
       {\path (A2_0) edge [->]node [auto,swap] {$\scriptstyle{\mathsf{C}(V_j, U_i)}$} (A2_2);}
    \path (A0_2) edge [->]node [auto] {$\,\scriptstyle{\boldrho_j}$} (A2_2);
      \slashing{\path (A0_2) -- node [rotate=90] {$\scriptscriptstyle{\slash}$} (A2_2);}{}
    \slashing
       {\path (A0_0.350) - -node [auto] {$\scriptstyle{\mathsf{A}(V_j, U_i)}$} (A0_2.190);
       \path (A0_0.350) edge [->] node [] {$\scriptscriptstyle{\slash}$} (A0_2.190);}
       {\path (A0_0) edge [->]node [auto] {$\scriptstyle{\mathsf{A}(V_j, U_i)}$} (A0_2);}
   \end{tikzpicture}
   \]
  Here, $\mathsf{C}_{ji}=\mathsf{C}(V_j, U_i)$ denotes  the module of concrete embeddings
   $\lambda\colon\tU_i\hookrightarrow \tV_j$ such that $\pi_j\circ\lambda=\pi_i$, where
   $\pi_i\colon\tU_i\to U_i$ and $\pi_j\colon\tV_j\to V_j$ are the quotient
   maps that are part of the chart structure for $U_i$ and $V_j$ respectively.

   We require that whenever $\rho_{ji}(e_j^{\red}\otimes\lambda)=\rho_{ji}(e_j^{\red}
   \otimes \lambda')$ for $\lambda,\lambda'\in\mathsf{A}_{ji}$, there is an element $g\in G_i$
   such that $\lambda\cdot g=\lambda'$.

 \item  For any $\mu_{ii'}$ in $\calO(\calU)$ with $U_i\subseteq V_j$, there are
   bimodule isomorphisms
   \[\boldalpha_{jii'}:\, \mathsf{A}_{ji}\otimes_{G_{i}}\ABST_{\mathfrak{U}}(\mu_{ii'})
   \Longrightarrow \mathsf{A}_{ji'}\]
   and
   \[\gamma_{jii'}:\, \mathsf{C}_{ji}\otimes_{G_i^{\red}}\CONCR_{\mathfrak{U}}(\mu_{ii'})
   \Longrightarrow\mathsf{C}_{ji'}.\]
 
 \item For any $\mu_{j'j}$ in $\calO(\calV)$ with $U_i\subseteq V_j$, there are bimodule
   isomorphisms
   \[\boldalpha_{j'ji}:\, \ABST_{\mathfrak{V}}(\mu_{j'j})\otimes_{G_{j}}\mathsf{A}_{ji}
   \Longrightarrow \mathsf{A}_{j'i}\]
   and
   \[\gamma_{j'ji}:\, \CONCR_{\mathfrak{V}}(\mu_{j'j})\otimes_{G_{j}^{\red}} \mathsf{C}_{ji}
   \Longrightarrow \mathsf{C}_{j'i}.\]

  \item These isomorphisms need to satisfy certain coherence laws and be compatible
   with the $\rho_{ii'}$s, $\rho_{ji}$s and $\rho_{j'j}$s in the following sense.  

   First, the $\boldalpha_{jii'}$ and $\gamma_{jii'}$ need to be compatible with the $\rho_{ii'}$,
   the $\rho_{ji'}$, and the $\rho_{ji}$ in the sense that the following composites are equal,

   \[
   \begin{tikzpicture}[xscale=2.3,yscale=-1.4]
    \node (A0_1) at (1, 0) {$G_i$};
    \node (A1_1) at (1, 0.6) {$\Downarrow\,\scriptstyle{\boldalpha_{jii'}}$};
    \node (A1_2) at (2, 1) {$G_j$};
    \node (A2_1) [rotate=225] at (0.75, 1.75) {$\Rightarrow$};
    \node (H2_1) at (1, 1.75) {$\scriptstyle{\rho_{ji'}}$};
    \node (A3_0) at (0, 2.5) {$G_{i'}^{\red}$};
    \node (A3_2) at (2, 2.5) {$G_j^{\red}$};
    \node (A1_0) at (0, 1) {$G_{i'}$};
    
    \def \z {-0.8}
    \node (C0_0) at (3+\z/2, 0.5) {$\equiv$};
    
    \node (A0_4) at (4.5+\z, 1) {$\Downarrow\,\scriptstyle{\rho_{ii'}}$};
    \node (A0_5) at (5+\z, 0) {$G_i$};
    \node (A0_6) at (5.5+\z, 1) {$\Downarrow\,\scriptstyle{\rho_{ji}}$};
    \node (A1_4) at (4+\z, 1) {$G_{i'}$};
    \node (A1_6) at (6+\z, 1) {$G_j$};
    \node (A2_5) at (5+\z, 1.5) {$G_i^{\red}$};
    \node (A3_4) at (4+\z, 2.5) {$G_{i'}^{\red}$};
    \node (A3_5) at (5+\z, 2.1) {$\Downarrow\,\scriptstyle{\gamma_{jii'}}$};
    \node (A3_6) at (6+\z, 2.5) {$G_j^{\red}$};
    \node (B0_0) at (4.4+\z, 1.75) {$\scriptstyle{\CONCR_{\mathfrak{U}}(\mu_{ii'})}$};
    \slashing{\path (A3_4) edge [->]node [rotate=25] {$\scriptscriptstyle{\slash}$} (A2_5);}
      {\path (A3_4) edge [->]node [rotate=25] {} (A2_5);}
    \node (B1_1) at (5.6+\z, 1.75) {$\scriptstyle{\mathsf{C}_{ji}}$};
    \slashing{\path (A2_5) edge [->]node [rotate=335] {$\scriptscriptstyle{\slash}$} (A3_6);}
      {\path (A2_5) edge [->]node [rotate=335] {} (A3_6);}
    \path (A0_5) edge [->] node [auto] {$\,\scriptstyle{\boldrho_i}$} (A2_5);
       \slashing{\path (A0_5) -- node [rotate=90] {$\scriptscriptstyle{\slash}$} (A2_5);}{}
    \path (A1_4) edge [->] node [auto,swap] {$\scriptstyle{\boldrho_{i'}}\,$} (A3_4);
       \slashing{\path (A1_4) -- node [rotate=90] {$\scriptscriptstyle{\slash}$} (A3_4);}{}
    \path (A0_1) edge [->] node [auto] {$\scriptstyle{\mathsf{A}_{ji}}$} (A1_2);
       \slashing{\path (A0_1) -- node [rotate=335] {$\scriptscriptstyle{\slash}$} (A1_2);}{}
    \path (A0_5) edge [->] node [auto] {$\scriptstyle{\mathsf{A}_{ji}}$} (A1_6);
       \slashing{\path (A0_5) -- node [rotate=335] {$\scriptscriptstyle{\slash}$} (A1_6);}{}
    \path (A1_0) edge [->] node [auto,swap] {$\scriptstyle{\boldrho_{i'}}\,$} (A3_0);
       \slashing{\path (A1_0) -- node [rotate=90] {$\scriptscriptstyle{\slash}$} (A3_0);}{}
    \path (A1_0) edge [->] node [auto] {$\scriptstyle{\ABST_{\mathfrak{U}}(\mu_{ii'})}$} (A0_1);
       \slashing{\path (A1_0) -- node [rotate=25] {$\scriptscriptstyle{\slash}$} (A0_1);}{}
    \path (A1_4) edge [->] node [auto] {$\scriptstyle{\ABST_{\mathfrak{U}}(\mu_{ii'})}$} (A0_5);
       \slashing{\path (A1_4) -- node [rotate=25] {$\scriptscriptstyle{\slash}$} (A0_5);}{}
    \path (A1_2) edge [->] node [auto] {$\,\scriptstyle{\boldrho_j}$} (A3_2);
       \slashing{\path (A1_2) -- node [rotate=90] {$\scriptscriptstyle{\slash}$} (A3_2);}{}
    \path (A1_6) edge [->] node [auto] {$\,\scriptstyle{\boldrho_j}$} (A3_6);
       \slashing{\path (A1_6) --node [rotate=90] {$\scriptscriptstyle{\slash}$} (A3_6);}{}
    \slashing
      {\path (A3_0.350) -- node [auto,swap] {$\scriptstyle{\mathsf{C}_{ji'}}$} (A3_2.190);
        \path (A3_0) edge [->] node [] {$\scriptscriptstyle{\slash}$} (A3_2);}
      {\path (A3_0) edge [->] node [auto,swap] {$\scriptstyle{\mathsf{C}_{ji}}$} (A3_2);}
    \slashing
       {\path (A3_4.350) -- node [auto,swap] {$\scriptstyle{\mathsf{C}_{ji'}}$} (A3_6.190);
         \path (A3_4) edge [->] node [] {$\scriptscriptstyle{\slash}$} (A3_6);}
       {\path (A3_4) edge [->] node [auto,swap] {$\scriptstyle{\mathsf{C}_{ji'}}$} (A3_6);}
    \slashing
      {\path (A1_0.350) -- node [auto,swap] {$\scriptstyle{\mathsf{A}_{ji'}}$} (A1_2.190);
        \path (A1_0) edge [->] node [] {$\scriptscriptstyle{\slash}$} (A1_2);}
      {\path (A1_0) edge [->] node [auto,swap] {$\scriptstyle{\mathsf{A}_{ji'}}$} (A1_2);}
   \end{tikzpicture}
   \]

   Also, the $\alpha_{j'ji}$ and $\gamma_{j'ji}$ need to be compatible with the $\rho_{ji}$,
   the $\rho_{j'i}$ and the $\rho_{j'j}$  in the sense that the following composites are equal,

   \[
   \begin{tikzpicture}[xscale=2.3,yscale=-1.4]
    \node (A0_1) at (1, 0) {$G_j$};
    \node (A1_1) at (1, 0.6) {$\Downarrow\,\scriptstyle{\boldalpha_{j'ji}}$};
    \node (A1_2) at (2, 1) {$G_{j'}$};
    \node (A2_1) [rotate=225] at (0.75, 1.75) {$\Rightarrow$};
    \node (H2_1) at (1, 1.75) {$\scriptstyle{\rho_{j'i}}$};
    \node (A3_0) at (0, 2.5) {$G_i^{\red}$};
    \node (A3_2) at (2, 2.5) {$G_{j'}^{\red}$};
    \node (A1_0) at (0, 1) {$G_i$};
    
    \def \z {-0.8}
    \node (C0_0) at (3+\z/2, 0.5) {$\equiv$};
   
    \node (A0_4) at (4.5+\z, 1) {$\Downarrow\,\scriptstyle{\rho_{ji}}$};
    \node (A0_5) at (5+\z, 0) {$G_j$};
    \node (A0_6) at (5.5+\z, 1) {$\Downarrow\,\scriptstyle{\rho_{j'j}}$};
    \node (A1_4) at (4+\z, 1) {$G_i$};
    \node (A1_6) at (6+\z, 1) {$G_{j'}$};
    \node (A2_5) at (5+\z, 1.5) {$G_j^{\red}$};
    \node (A3_4) at (4+\z, 2.5) {$G_i^{\red}$};
    \node (A3_5) at (5+\z, 2.1) {$\Downarrow\,\scriptstyle{\gamma_{j'ji}}$};
    \node (A3_6) at (6+\z, 2.5) {$G_{j'}^{\red}$};
    \node (B0_0) at (4.4+\z, 1.75) {$\scriptstyle{\mathsf{C}_{ji}}$};
    \slashing{\path (A3_4) edge [->]node [rotate=25] {$\scriptscriptstyle{\slash}$} (A2_5);}
      {\path (A3_4) edge [->]node [rotate=25] {} (A2_5);}
    \node (B1_1) at (5.6+\z, 1.75) {$\scriptstyle{\CONCR_{\mathfrak{V}}(\mu_{j'j})}$};
    \slashing{\path (A2_5) edge [->]node [rotate=335] {$\scriptscriptstyle{\slash}$} (A3_6);}
      {\path (A2_5) edge [->]node [rotate=335] {} (A3_6);}
    \path (A0_5) edge [->] node [auto] {$\,\scriptstyle{\boldrho_j}$} (A2_5);
       \slashing{\path (A0_5) -- node [rotate=90] {$\scriptscriptstyle{\slash}$} (A2_5);}{}
    \path (A1_4) edge [->] node [auto,swap] {$\scriptstyle{\boldrho_i}\,$} (A3_4);
       \slashing{\path (A1_4) -- node [rotate=90] {$\scriptscriptstyle{\slash}$} (A3_4);}{}
    \path (A0_1) edge [->] node [auto] {$\scriptstyle{\ABST_{\mathfrak{V}}(\mu_{j'j})}$} (A1_2);
       \slashing{\path (A0_1) -- node [rotate=335] {$\scriptscriptstyle{\slash}$} (A1_2);}{}
    \path (A0_5) edge [->] node [auto] {$\scriptstyle{\ABST_{\mathfrak{V}}(\mu_{j'j})}$} (A1_6);
       \slashing{\path (A0_5) -- node [rotate=335] {$\scriptscriptstyle{\slash}$} (A1_6);}{}
    \path (A1_0) edge [->] node [auto,swap] {$\scriptstyle{\boldrho_i}\,$} (A3_0);
       \slashing{\path (A1_0) -- node [rotate=90] {$\scriptscriptstyle{\slash}$} (A3_0);}{}
    \path (A1_0) edge [->] node [auto] {$\scriptstyle{\mathsf{A}_{ji}}$} (A0_1);
       \slashing{\path (A1_0) -- node [rotate=25] {$\scriptscriptstyle{\slash}$} (A0_1);}{}
    \path (A1_4) edge [->] node [auto] {$\scriptstyle{\mathsf{A}_{ji}}$} (A0_5);
       \slashing{\path (A1_4) -- node [rotate=25] {$\scriptscriptstyle{\slash}$} (A0_5);}{}
    \path (A1_2) edge [->] node [auto] {$\,\scriptstyle{\boldrho_{j'}}$} (A3_2);
       \slashing{\path (A1_2) -- node [rotate=90] {$\scriptscriptstyle{\slash}$} (A3_2);}{}
    \path (A1_6) edge [->] node [auto] {$\,\scriptstyle{\boldrho_{j'}}$} (A3_6);
       \slashing{\path (A1_6) --node [rotate=90] {$\scriptscriptstyle{\slash}$} (A3_6);}{}
    \slashing
      {\path (A3_0.350) -- node [auto,swap] {$\scriptstyle{\mathsf{C}_{j'i}}$} (A3_2.190);
        \path (A3_0) edge [->] node [] {$\scriptscriptstyle{\slash}$} (A3_2);}
      {\path (A3_0) edge [->] node [auto,swap] {$\scriptstyle{\mathsf{C}_{j'i}}$} (A3_2);}
    \slashing
       {\path (A3_4.350) -- node [auto,swap] {$\scriptstyle{\mathsf{C}_{j'i}}$} (A3_6.190);
         \path (A3_4) edge [->] node [] {$\scriptscriptstyle{\slash}$} (A3_6);}
       {\path (A3_4) edge [->] node [auto,swap] {$\scriptstyle{\mathsf{C}_{j'i}}$} (A3_6);}
    \slashing
      {\path (A1_0.350) -- node [auto,swap] {$\scriptstyle{\mathsf{A}_{j'i}}$} (A1_2.190);
        \path (A1_0) edge [->] node [] {$\scriptscriptstyle{\slash}$} (A1_2);}
      {\path (A1_0) edge [->] node [auto,swap] {$\scriptstyle{\mathsf{A}_{j'i}}$} (A1_2);}
   \end{tikzpicture}
   \]

   For arrows $\mu_{i'i''},\mu_{ii'}\in\calO(\calU)$ and $\mu_{j'j},\mu_{j''j'}\in\calO(\calV)$
   with $U_i\subseteq V_j$, the following generalized associativity pentagons need to commute:

   \[
   \begin{tikzpicture}[xscale=3,yscale=-1.2, implies/.style={-implies, double, double distance = 1.5}]
    \node (A0_1) at (1, 0) {$\mathsf{A}_{ji}\otimes(\ABST_{\mathfrak{U}}(\mu_{ii'})\otimes\ABST_{\mathfrak{U}}(\mu_{i'i''}))$};
    \node (A1_0) at (0, 1) {$(\mathsf{A}_{ji}\otimes\ABST_{\mathfrak{U}}(\mu_{ii'}))\otimes\ABST_{\mathfrak{U}}(\mu_{i'i''})$};
    \node (A1_2) at (2, 1) {$\mathsf{A}_{ji}\otimes\ABST_{\mathfrak{U}}(\mu_{ii''})$};
    \node (A2_0) at (0, 2) {$\mathsf{A}_{ji'}\otimes\ABST_{\mathfrak{U}}(\mu_{i'i''})$};
    \node (A2_2) at (2, 2) {$\mathsf{A}_{ji''}$};
    
    \node (B0_0) at (1.9, 0.45) {$\scriptstyle{\mathsf{A}_{ji}\otimes\boldalpha_{ii'i''}}$};
    \path (A0_1) edge [implies]node [auto] {} (A1_2);
    \path (A1_2) edge [implies]node [auto] {$\scriptstyle{\boldalpha_{jii''}}$} (A2_2);
    \path (A0_1) edge [double, double equal sign distance]node [auto,swap] {$\scriptstyle{\sim}$} (A1_0);
    \path (A1_0) edge [implies]node [auto,swap] {$\scriptstyle{\boldalpha_{jii'}\otimes\ABST_{\mathfrak{U}}(\mu_{i'i''})}$} (A2_0);    
    \path (A2_0) edge [implies]node [auto,swap] {$\scriptstyle{\boldalpha_{ji'i''}}$} (A2_2);
   \end{tikzpicture}
   \]

   \[
   \begin{tikzpicture}[xscale=3,yscale=-1.2, implies/.style={-implies, double, double distance = 1.5}]
    \node (A0_1) at (1, 0) {$\ABST_{\mathfrak{V}}(\mu_{j'j})\otimes(\mathsf{A}_{ji}\otimes\ABST_{\mathfrak{U}}(\mu_{ii'}))$};
    \node (A1_0) at (0, 1) {$(\ABST_{\mathfrak{V}}(\mu_{j'j})\otimes\mathsf{A}_{ji})\otimes\ABST_{\mathfrak{U}}(\mu_{ii'})$};
    \node (A1_2) at (2, 1) {$\ABST_{\mathfrak{V}}(\mu_{j'j})\otimes\mathsf{A}_{ji'}$};
    \node (A2_0) at (0, 2) {$\mathsf{A}_{j'i}\otimes\ABST_{\mathfrak{U}}(\mu_{ii'})$};
    \node (A2_2) at (2, 2) {$\mathsf{A}_{j'i'}$};
    
    \node (B0_0) at (2, 0.45) {$\scriptstyle{\ABST_{\mathfrak{V}}(\mu_{j'j})\otimes\boldalpha_{jii'}}$};
    \path (A0_1) edge [implies]node [auto] {} (A1_2);
    \path (A1_2) edge [implies]node [auto] {$\scriptstyle{\boldalpha_{j'ji'}}$} (A2_2);
    \path (A0_1) edge [double, double equal sign distance]node [auto,swap] {$\scriptstyle{\sim}$} (A1_0);
    \path (A1_0) edge [implies]node [auto,swap] {$\scriptstyle{\boldalpha_{j'ji}\otimes\ABST_{\mathfrak{U}}(\mu_{ii'})}$} (A2_0);    
    \path (A2_0) edge [implies]node [auto,swap] {$\scriptstyle{\boldalpha_{j'ii'}}$} (A2_2);
   \end{tikzpicture}
   \]

   \[
   \begin{tikzpicture}[xscale=3,yscale=-1.2, implies/.style={-implies, double, double distance = 1.5}]
    \node (A0_1) at (1, 0) {$\ABST_{\mathfrak{V}}(\mu_{j''j'})\otimes(\ABST_{\mathfrak{V}}(\mu_{j'j})\otimes\mathsf{A}_{ji})$};
    \node (A1_0) at (0, 1) {$(\ABST_{\mathfrak{V}}(\mu_{j''j'})\otimes\ABST_{\mathfrak{V}}(\mu_{j'j}))\otimes \mathsf{A}_{ji}$};
    \node (A1_2) at (2, 1) {$\ABST_{\mathfrak{V}}(\mu_{j''j'})\otimes\mathsf{A}_{j'i}$};
    \node (A2_0) at (0, 2) {$\ABST_{\mathfrak{V}}(\mu_{j''j})\otimes\mathsf{A}_{ji}$};
    \node (A2_2) at (2, 2) {$\mathsf{A}_{j''i}$};
    
    \node (B0_0) at (2, 0.45) {$\scriptstyle{\ABST_{\mathfrak{V}}(\mu_{j''j'})\otimes\boldalpha_{j'ji}}$};
    \path (A0_1) edge [implies]node [auto] {} (A1_2);
    \path (A1_2) edge [implies]node [auto] {$\scriptstyle{\boldalpha_{j''j'i}}$} (A2_2);
    \path (A0_1) edge [double, double equal sign distance]node [auto,swap] {$\scriptstyle{\sim}$} (A1_0);
    \path (A1_0) edge [implies]node [auto,swap] {$\scriptstyle{\boldalpha_{j''j'j}\otimes\mathsf{A}_{ji}}$} (A2_0);    
    \path (A2_0) edge [implies]node [auto,swap] {$\scriptstyle{\boldalpha_{j''ji}}$} (A2_2);
   \end{tikzpicture}
   \]

   Whenever $U_i\subseteq V_j$, the following unit coherence diagrams need to commute:

   \[
   \begin{tikzpicture}[xscale=3.4,yscale=-1.2, implies/.style={-implies, double, double distance = 1.5}]
    \node (A0_0) at (0, 0) {$\mathsf{A}_{ji}\otimes\mathsf{G}_i$};
    \node (A0_1) at (1, 0) {$\mathsf{A}_{ji}\otimes\ABST_{\mathfrak{U}}(\mu_{ii})$};
    \node (A1_1) at (1, 1) {$\mathsf{A}_{ji}$};
        
    \path (A0_0) edge [implies] node [auto] {$\scriptstyle{\mathsf{A}_{ji}\otimes\boldalpha_i}$} (A0_1);
    \path (A0_0) edge [double, double equal sign distance]node [auto,swap] {$\scriptstyle{\sim}$} (A1_1);
    \path (A0_1) edge [implies]node [auto] {$\scriptstyle{\boldalpha_{jii}}$} (A1_1);
    
    \def \z {1}
    \node (B0_0) at (1+3*\z/5, 0.5) {$\mbox{and}$}; 
    
    \node (C0_0) at (1+\z, 0) {$\mathsf{G}_j\otimes\mathsf{A}_{ji}$};
    \node (C0_1) at (2+\z, 0) {$\ABST_{\mathfrak{V}}(\mu_{jj})\otimes\mathsf{A}_{ji}$};
    \node (C1_1) at (2+\z, 1) {$\mathsf{A}_{ji}$};
        
    \path (C0_0) edge [implies] node [auto] {$\scriptstyle{\boldalpha_j\otimes\mathsf{A}_{ji}}$} (C0_1);
    \path (C0_0) edge [double, double equal sign distance]node [auto,swap] {$\scriptstyle{\sim}$} (C1_1);
    \path (C0_1) edge [implies]node [auto] {$\scriptstyle{\boldalpha_{jji}}$} (C1_1);
   \end{tikzpicture}
   \]
   where the $\sim$ indicates the canonical isomorphism.

   Analogous associativity and unit coherence conditions apply to the $\gamma_{jii'}$
   and $\gamma_{j'ji}$.
\end{enumerate}
\end{dfn}

With this definition, it is easy to check that refinements still have the following transitive
property.

\begin{lma}
Suppose we have orbifold atlases $\mathfrak{U}$, $\mathfrak{V}$, and $\mathfrak{W}$
for the same topological space $X$, and $\mathfrak{U}$ is a refinement of $\mathfrak{V}$ and
$\mathfrak{V}$ is a refinement of $\mathfrak{W}$.   Then $\mathfrak{U}$ is a refinement of
$\mathfrak{W}$.
\end{lma}
 
The notion of atlas equivalence now generalizes as follows:

\begin{dfn}\label{D:atlasequiv}
Two orbifold atlases $\mathfrak{V}$ and $\mathfrak{W}$ for the same underlying space $X$ are
\emph{equivalent} if they have a common refinement $\mathfrak{U}$.
\end{dfn}

This relation is clearly symmetric and reflexive.
It is also transitive; this will follow from Theorems~\ref{T:morita-atlas}
and~\ref{T:morita-inverse} proved below (and the fact that being Morita equivalent
is an equivalence relation).

\begin{rmks}\label{refinement-rmks}
\begin{enumerate}
 \item\label{D:rmk1}
  The argument given in Lemma~\ref{tilde-l} applies here showing that each $\rho_{ji}$ in
  part (\ref{D:modules}) of Definition~\ref{refinement} gives rise to a module map
  $\trho_{ji}\colon\mathsf{A}_{ji}\to\mathsf{C}_{ji}$.

  \item\label{strong_compatibility_2}
   It is clear from the set-up of the definition of refinement that the union of an orbifold atlas 
   $\mathfrak{V}$ with a refinement $\mathfrak{U}$ gives rise to a canonical new atlas structure,
   except for the fact that any charts which occur in both atlases will occur twice in the new
   structure.   Because of this, we obtain the following generalization of the strong compatibility
   result for atlas charts (Lemma~\ref{strong_compatibility}):

   {\em For any charts $\tU_1$ in $\calU$ and $\tV_2,\tV_3$ in $\calV$ with $U_1\subseteq V_3$
   and $V_2\subseteq V_3$ and for any points $x_1\in\tU_1$, $y_2\in\tV_2$ and $y_3\in\tV_3$ and abstract
   embeddings $\lambda_{31}\in\mathsf{A}_{31}$ and $\lambda_{32}\in\ABST_{\mathfrak{V}}(\mu_{32})$,
   such that $\trho_{31}(\lambda_{31})(x_1)=y_3$ and $\trho_{32}(\lambda_{32})(y_2)=y_3$,
   there is a chart $\tU_0$ in $\calU$ with a point $x_0\in\tU_0$ and abstract embeddings
   $\kappa_{10}\in\ABST_{\mathfrak{U}}(\mu_{10})$ and $\kappa_{20}\in\mathsf{A}_{20}$,
   such that $\trho_{10}(\kappa_{10})(x_0)=x_1$, $\trho_{20}(\kappa_{20})(x_0)=y_2$
   and $\boldalpha_{310}(\lambda_{31}\otimes\kappa_{10})=
   \boldalpha_{320}(\lambda_{32}\otimes\kappa_{20})$.}
\end{enumerate}
\end{rmks}

\subsection{Equivalence of Smooth Groupoids}

In this section we summarize the description of equivalence of proper \'etale groupoids in terms of
Hilsum-Skandalis maps~\cite{Hae,Pr}.

\begin{dfn}
Let $\calG$ be a topological groupoid.  A \emph{right ${\calG}$-bundle} over a manifold $X$ is
a manifold $M$ with smooth maps

\[
\begin{tikzpicture}[xscale=0.7,yscale=-0.6]
    \node (A0_0) at (0, 0) {$M$};
    \node (A0_2) at (2, 0) {$X$};
    \node (A2_0) at (0, 2) {$\calG_0$};
    \path (A0_0) edge [->]node [auto,swap] {$\scriptstyle{\Gbmap}$} (A2_0);
    \path (A0_0) edge [->]node [auto] {$\scriptstyle{\Hbmap}$} (A0_2);
\end{tikzpicture}
\]
and a right ${\calG}$-action $\mu$ on  $M$, with anchor map $\Gbmap \colon M\rightarrow \calG_0$,
such that  $\Hbmap(mg) = \Hbmap(m)$ ($\calG$ acts on the fibres of $\Hbmap$)
and $\Gbmap(m g)  =s(g)$ for any $m\in M$ and any $g\in \calG_1$  with $\Gbmap(m)=t(g)$.   

Such a bundle $M$ is \emph{principal} if

\begin{enumerate}
 \item $\Hbmap$ is a surjective submersion, and 
 
 \item the map $(pr_1, \mu) \colon M \times_{\calG_0} \calG_1\rightarrow M \times_X M$, sending
  $(m, g)$ to $(m,mg)$, is a diffeomorphism.
\end{enumerate}
\end{dfn}

\begin{dfn}
A {\em Hilsum-Skandalis map} $\calG\nrightarrow \calH$ is represented by a right
$\calG$-bundle $M$ over $\calH_0$, $\calG_0\stackrel{\Gbmap}{\longleftarrow} M
\stackrel{\Hbmap}{\longrightarrow}\calH_0$, which is also a principal left $\calH$-bundle 
over $\Gbmap$, and such that the left and right actions commute.    So we have that
\begin{eqnarray*}
& \Gbmap(hm)=\Gbmap(m)\,\mbox{ and }\,\Hbmap(mg)=\Hbmap(m) & \\
& h(mg)=(hm)g & \\
& \Gbmap(m g)  =s(g)\,\mbox{ and  }\,\Hbmap(h m) = t(h) &
\end{eqnarray*}
for any $m\in M$, $g\in \calG_1$ and $h\in \calH_1$ with $s(h)=\Hbmap(m)$ and $t(g)=\Gbmap(m)$. 

Moreover, since the $\calH$-bundle is principal, $\Hbmap$ is a surjective submersion,
and the map $\calH_1  \times_{\calH_0} M  \rightarrow M \times_{\calG_0} M$ is a diffeomorphism. 
\end{dfn}

\begin{dfn}
A Hilsum-Skandalis map is a \emph{Morita equivalence} when it is both a principal left
$\calG$-bundle and a principal right $\calH$-bundle.
\end{dfn}

Two proper \'etale groupoids are Morita equivalent if there is a Morita equivalence between them.
Note that in particular, if $\calG$ and $\calH$ are Morita equivalent,
we get an induced homeomorphism between their quotient
spaces $\calG_0/\calG_1 \cong \calH_0/\calH_1$.  

\subsection{Morita Equivalence Implies Atlas Equivalence}

The goal of this section is to prove the following:  

\begin{thm}\label{T:morita-atlas}
Let $X$ be a space, and $\calG$ and $\calH$ be proper \'etale groupoids such that both $\calG_0/\calG_1$
and $\calH_0/\calH_1$ are homeomorphic to $X$.   If $\calG$ and $\calH$ are Morita equivalent,
and $\mathfrak{V}$ and $\mathfrak{W}$ are orbifold atlases for $X$ constructed from the
translation neighbourhoods in $\calG$ and $\calH$ respectively, then 
$\mathfrak{V}$ and $\mathfrak{W}$ are equivalent in the sense of Definition~\ref{D:atlasequiv}.
\end{thm}

This result implies that any two atlases obtained from the same groupoid (as in
Section~\ref{S:atlascon}) are equivalent as orbifold atlases:

\begin{cor}
If $\calG$ is a proper \'etale groupoid and $\mathfrak{V}$ and $\mathfrak{W}$ are two orbifold atlases
constructed from translation neighbourhoods in $\calG$, then $\mathfrak{V}$ and $\mathfrak{W}$
are equivalent atlases for the quotient space  $\calG_0/\calG_1$.
\end{cor}

We will prove Theorem \ref{T:morita-atlas} in several stages.  We begin by fixing notation.   Throughout this
subsection, $\calG$ and $\calH$ will denote proper \'etale groupoids
with a Hilsum-Skandalis Morita equivalence
\[M\colon \calG \nrightarrow \calH\]
with bundle maps
$\calG_0\stackrel{\Gbmap}{\longleftarrow} M\stackrel{\Hbmap}{\longrightarrow}\calH_0$.
Let $X$ be a space that is homeomorphic to  $\calG_0/\calG_1 \cong \calH_0/\calH_1$.   
Furthermore, $\mathfrak{V}$ and $\mathfrak{W}$ will denote induced atlases consisting of
translation neighbourhoods for $X$ in $\calG$ and $\calH$ respectively.    Note that the maps
from the charts in $\mathfrak{V}$ into the underlying space $X$ are obtained as restrictions
of the composition of the projection map $\calG_0 \to \calG_0 /\calG_1$ with the isomorphism
$\calG_0 /\calG_1 \cong X$; similarly for $\mathfrak{W}$.  

Since $\mathfrak{V}$ consists of translation neighbourhoods in $\calG$, we can write
$s^{-1}(\tV)\cap t^{-1}(\tV)\cong G_V\times \tV$ for all $\tV \in \calV$, 
and similarly for $\mathfrak{W}$: 
$s^{-1}(\tW)\cap t^{-1}(\tW)\cong H_W\times \tW$ for all $\tW \in \calW$.    
Note that the structure groups $G_V$ and
$H_W$ are all finite and discrete.   We will prove below that 
for any $m \in M$ we can choose an open neighbourhood $S_m$ containing $m$, 
such that $\Gbmap|_{S_m}\colon S_m\to\Gbmap(S_m)$ and
$\Hbmap|_{S_m}\colon S_m\to\Hbmap(S_m)$  are diffeomorphisms and 
$\Gbmap(S_m)$ and $\Hbmap(S_m)$ are invariant subsets of $\tV$ and $\tW$ respectively.
Explicitly, this means that 
for each $g\in G_V$ either $g(\Gbmap(S_m))=\Gbmap(S_m)$ or
$g(\Gbmap(S_m))\cap\Gbmap(S_m)=\varnothing$, and analogously for each $h\in H_W$ and $\Hbmap(S_m)$.  
We will call any such neighbourhood $S_m$ an \emph{invariant local bisection} of $M$.

Our plan is to show that these invariant local bisections can be used to create an atlas $\mathfrak{U}$
which is a common refinement of $\mathfrak{V}$ and $\mathfrak{W}$.  
Our first proposition shows that there are enough of these invariant local bisections. 

\begin{prop}\label{prop-final}
The invariant local bisections form a basis for the 
topology on the space $\Gbmap^{-1}(\bigcup_{\tV\in\,\calV}\tV)\cap
\Hbmap^{-1}(\bigcup_{\tW\in\calW}\tW)\subseteq M$.
\end{prop}

\begin{proof}
First, since $\calG$ is \'etale and $M\stackrel{\varepsilon}{\to}\calH_0$ is a right principal $\calG$-bundle, 
the map $\varepsilon$ is a submersion with discrete fibers, so $\varepsilon$ is \'etale.
Similarly, $\tau\colon M\to \calG_0$ is \'etale (it is a left principal $\calH$-bundle and $\calH$ is \'etale).
Hence, open subsets $S\subseteq M$ for which both the restriction of $\varepsilon$ and the restriction of $\tau$
are diffeomorphisms form a basis for the topology.

Since $\varepsilon$ is a submersion, we may consider the pullback groupoid $$
\varepsilon^*\calH:=\left(M\pullback{\varepsilon}{t}\calH_1\pullback{s}{\varepsilon}M\rightrightarrows M\right)$$
and this is again an \'etale groupoid since $s$, $t$ and $\varepsilon$ are \'etale maps.
Also, the functor
\[
\begin{tikzpicture}[xscale=1.5, yscale= -0.7]
    \node (A0_0) at (0, 0) {$M\,_{\varepsilon}\hspace{-1mm}\times_t\,\mathcal{H}_1\,_s\hspace{-1mm}\times_{\varepsilon}\,M$};
    \node (A2_0) at (0, 2) {$M$};
    \node (A0_2) at (2, 0) {$\mathcal{H}_1$};
    \node (A2_2) at (2, 2) {$\mathcal{H}_0$};
    
    \draw[transform canvas={xshift=0.5ex},->] (A0_0) -- (A2_0);
    \draw[transform canvas={xshift=-0.5ex},->] (A0_0) -- (A2_0);
    \draw[transform canvas={xshift=0.5ex},->] (A0_2) -- (A2_2);
    \draw[transform canvas={xshift=-0.5ex},->] (A0_2) -- (A2_2);
    \path (A0_0) edge [->]node [auto] {$\scriptstyle{\pi}$} (A0_2);
    \path (A2_0) edge [->]node [auto,swap] {$\scriptstyle{\varepsilon}$} (A2_2);
\end{tikzpicture}
\]
is clearly a Morita equivalence. Thus, for any open $S\subseteq M$ such that $\varepsilon|_S$ is a diffeomorphism, with $W=\varepsilon(S)$, the induced map between the restricted groupoids,
\[
\begin{tikzpicture}[xscale=1.5, yscale= -0.7]
    \node (A0_0) at (0, 0) {$S\,_{\varepsilon}\hspace{-1mm}\times_t\,\mathcal{H}_1\,_s\hspace{-1mm}\times_{\varepsilon}\,S$};
    \node (A2_0) at (0, 2) {$S$};
    \node (A0_2) at (2, 0) {${\mathcal{H}_1}|_{W}$};
    \node (A2_2) at (2, 2) {$W$};
    
    \draw[transform canvas={xshift=0.5ex},->] (A0_0) -- (A2_0);
    \draw[transform canvas={xshift=-0.5ex},->] (A0_0) -- (A2_0);
    \draw[transform canvas={xshift=0.5ex},->] (A0_2) -- (A2_2);
    \draw[transform canvas={xshift=-0.5ex},->] (A0_2) -- (A2_2);
    \path (A0_0) edge [->]node [auto] {$\scriptstyle{\pi}$} (A0_2);
    \path (A2_0) edge [->]node [auto,swap] {$\scriptstyle{\varepsilon}$} (A2_2);
\end{tikzpicture}
\]
is an isomorphism of smooth groupoids. Hence, $S$ is a translation open subset for $\varepsilon^*\calH$ if and only if $W$ is a translation open subset for $\calH$.
So the translation open subsets form a basis for the topology of $M$.

By symmetry, similar results hold for $\tau$ and $\calG$. Furthermore, by bi-principality of the Hilsum-Skandalis bundle, there is a canonical isomorphism of Lie groupoids, $\varepsilon^*\calH\cong\tau^*\calG$, given by $(m',h,m)\mapsto (m',g,m)$, where $g$ and $h$ uniquely determine each other through the equation $m'g=hm$. So any open subset $S\subseteq M$ is translation open for $\varepsilon^*\calH$ if and only if it is translation open for $\tau^*\calG$. This gives us the required result.
\end{proof}

Now fix any collection $\{S_i\}$ of invariant local bisections which forms a basis for the topology on
$\Gbmap^{-1}(\bigcup_{\tV\,\in\,\calV}\tV)\cap\Hbmap^{-1}(\bigcup_{\tW\in\,\calW}\tW)\subseteq M$.
Then we define the atlas $\mathfrak{U}$ by taking all the charts $U_i:=\Gbmap(S_i)$.
Note that each $U_i$ is a translation neighbourhood in $\calG$, so we obtain the
remainder of the atlas structure from $\calG$ as described in Section~\ref{S:atlascon}.
In particular, if $S_i$ is constructed around a point $m\in\Gbmap^{-1}(V)\cap\Hbmap^{-1}(W)$
(as in Proposition~\ref{prop-final}), then
$U_i$ gets the structure group $G_{U_i}\le G_V$ consisting of all the
elements $g\in G_V$ such that $g(U_i)=U_i$.

\begin{rmks} \label{samegp} We have arbitrarily chosen to
define our atlas structure in terms of $\calG$;   we would obtain the same structure groups
if we defined it in terms of $\calH$, since 
the subgroup $G_U$ has a free and transitive right action on the
connected  components of this subspace, whereas a corresponding subgroup $H_U$ has a free and transitive
left action on the set of connected components.  These actions commute in the sense that for
any connected component $C$, $h(C g)=(h C) g$, and so the 
groups are isomorphic.    
\end{rmks}

The following will complete the proof of  Theorem~\ref{T:morita-atlas}.  

\begin{prop}
The atlas $\mathfrak{U}$ is a common refinement for $\mathfrak{V}$ and $\mathfrak{W}$.  
\end{prop}

\begin{proof}
The atlas bimodules $\mathsf{A}^{\mathfrak{V}}_{ji}\colon G_{U_i}\longarrownot\longrightarrow  G_{V_j}$
and $\mathsf{C}^{\mathfrak{V}}_{ji}\colon G_{U_i}^{\red}\longarrownot\longrightarrow G_{V_j}^{\red}$,  
together with the required module homomorphisms $\rho_{ji}$, $\alpha_{jii'}$ and $\gamma_{jii'}$
can be obtained from the structure of the groupoid
$\calG$ in the same way as the atlas structure of $\mathfrak{U}$ is defined, following the method
of Section~\ref{S:atlascon}.
The required family of coherent isomorphisms is provided by $\calG$ as well:
in fact, all the structure is there to make $\mathfrak{U}\cup \mathfrak{V}$ an orbifold atlas,
since we chose to use $\calG$ in defining $\mathfrak{U}$.  

To make $\mathfrak{U}$ a refinement for $\mathfrak{W}$, we need to define atlas bimodules
$\mathsf{A}^{\mathfrak{W}}_{ki}\colon G_{U_i}\longarrownot\longrightarrow H_{W_k}$
and $\mathsf{C}^{\mathfrak{W}}_{ki}\colon G_{U_i}^{\red}\longarrownot\longrightarrow H_{W_k}^{\red}$
together with $\rho_{ki}$ whenever the map 
$\Gbmap\colon \Hbmap^{-1}(\tW_k)\cap\Gbmap^{-1}(U_i) \to U_i$
is surjective (that is, whenever $U_i$ is a subset of the chart 
$\Gbmap(\Hbmap^{-1}(\tW_k))$ in $\calG$, corresponding to the chart $\tW_k$ of $\calH$).
We define  $\mathsf{A}^{\mathfrak{W}}_{ki}$ as the set of connected components of
$\Hbmap^{-1}(\tW_k)\cap \Gbmap^{-1}(U_i)\subseteq M$;   the group $H_{W_k}$ acts
freely and transitively and $G_{U_i}$ acts freely on this set, so we obtain an atlas bimodule.
The bimodule $\mathsf{C}^{\mathfrak{W}}_{ki}$ on the concrete level is obtained as follows:
for each component $T\subseteq \Hbmap^{-1}(\tW_k)\cap\Gbmap^{-1}(U_i)$,
we define the concrete embedding $\lambda_T$ as the composite

\[
\begin{tikzpicture}[xscale=1.8,yscale=-1.2]
    \node (A0_0) at (0, 0) {$U_i$};
    \node (A0_1) at (1.2, 0) {$T$};
    \node (A0_2) at (2, 0) {$\tW_k$.};
    \path (A0_0) edge [->]node [auto] {$\scriptstyle{(\Gbmap|_{\scriptstyle{T}})^{-1}}$} (A0_1);
    \path (A0_1) edge [->]node [auto] {$\scriptstyle{\Hbmap}$} (A0_2);
\end{tikzpicture}
\]
and we define $\mathsf{C}^{\mathfrak{W}}_{ki}$ as the set of all such $\lambda_T$'s.  
The 2-cell $\rho_{ki}$ is defined by sending $e\otimes T$ to $\lambda_T\otimes e$.

For $\mu_{k'k}$ in $\calO(\calW)$, we define the isomorphism $\boldalpha_{k'ki}\colon
\ABST_{\mathfrak{W}}(\mu_{k'k})\otimes_{H_{W_k}}\mathsf{A}^{\mathfrak{W}}_{ki}\to
\mathsf{A}^{\mathfrak{W}}_{k'i}$
as follows.   Recall that the elements of $\ABST_{\mathfrak{W}}(\mu_{k'k})$ are the connected
components of $s^{-1}(\tW_k)\cap t^{-1}(\tW_{k'})$.
Note that the map $\Hbmap$ restricted to any component $T$ of
$\Hbmap^{-1}(\tW_k)\cap\Gbmap^{-1}(U_i)$ is a diffeomorphism, as is
the source map when restricted to any component of $s^{-1}(\tW_k)\cap t^{-1}(\tW_{k'})$.
So the right action of $\calH$ on $M$ induces a well-defined map
$\ABST_{\mathfrak{W}}(\mu_{k'k})\times\mathsf{A}^{\mathfrak{W}}_{ki}
\to\mathsf{A}^{\mathfrak{W}}_{k'i}$.
Furthermore, this is well-defined with respect to the action of $H_{W_k}$ so that this factors
through the required isomorphism  $\boldalpha_{k'ki}\colon \ABST_{\mathfrak{W}}(\mu_{k'k})
\otimes _{H_{W_k}}\mathsf{A}^{\mathfrak{W}}_{ki}\to\mathsf{A}^{\mathfrak{W}}_{k'i}$.
The $\gamma_{k'ki}$ are defined by composition as usual, and the required diagrams involving the
$\rho_{ki}$ are easily seen to commute.

A straightforward calculation shows that these $\boldalpha_{k'ki}$ also satisfy the other
required coherence conditions.  This can also be seen by the following observation:
the atlas $\mathfrak{U}$ can be viewed as an atlas induced by $\calH$ (rather than by $\calG$ as
constructed above), by viewing each $U_i$ as embedded in $\tW_k$ by the embedding
$\lambda_T$, and considering the structure group to be $H_{U_i}$ via an  isomorphism of this group with $G_{U_i}$ as  in Remark \ref{samegp}.     The abstract and concrete modules
defined by $\calG$ can be translated into abstract and concrete modules defined by $\calH$
through the actions of both groupoids on the invariant local bisections of $M$.
\end{proof}

\subsection{Atlas Equivalence Implies Morita Equivalence}
In this section, we prove the converse of Theorem \ref{T:morita-atlas}.  

\begin{thm}\label{T:morita-inverse}
If $\mathfrak{V}$ and $\mathfrak{W}$ are equivalent atlases for a space $X$, 
then the induced groupoids $\calG(\mathfrak{V})$ and $\calG(\mathfrak{W})$ are Morita equivalent.
\end{thm}

\begin{proof}
Since $\mathfrak{V}$ and $\mathfrak{W}$ are equivalent, they have a common refinement
$\mathfrak{U}$ as in Definition~\ref{refinement}. 
Let $I$, $J$ and $K$ be the index sets for $\calU$, $\calV$ and $\calW$ respectively.
Write $\mathsf{A}^{\mathfrak{V}}_{ji}$ and
$\mathsf{C}^{\mathfrak{V}}_{ji}$ (respectively, $\mathsf{A}^{\mathfrak{W}}_{ji}$ and
$\mathsf{C}^{\mathfrak{W}}_{ji}$) for the atlas bimodules establishing $\mathfrak{U}$ as a
refinement of $\mathfrak{V}$ (respectively, of $\mathfrak{W}$).   We will use these bimodules
to construct a Morita equivalence of groupoids.  This will be given by a Hilsum-Skandalis map
\[M(\mathfrak{U})\colon\calG(\mathfrak{V})\nrightarrow\calG(\mathfrak{W}),\]
with surjective submersions
\begin{tikzpicture}[baseline, xscale=1.8,yscale=-1.2]
    \node (A0_0) at (0, -0.1) {$\calG(\mathfrak{V})_0$};
    \node (A0_1) at (1, -0.1) {$M(\mathfrak{U})$};
    \node (A0_2) at (2, -0.1) {$\calG(\mathfrak{W})_0$,};
    \path (A0_1) edge [->]node [auto,swap] {$\scriptstyle{\Gbmap}$} (A0_0);
    \path (A0_1) edge [->]node [auto] {$\scriptstyle{\Hbmap}$} (A0_2);
\end{tikzpicture}
such that the actions induce diffeomorphisms

\[M(\mathfrak{U})\times_{\calG(\mathfrak{W})_0}\calG(\mathfrak{W})_1\cong
M(\mathfrak{U})\times_{\calG(\mathfrak{V})_0}M(\mathfrak{U})\]
and
\[M(\mathfrak{U})\times_{\calG(\mathfrak{V})_0}\calG(\mathfrak{V})_1\cong
M(\mathfrak{U})\times_{\calG(\mathfrak{W})_0}M(\mathfrak{U}).\]

Note that since the charts $\calV$ of $\mathfrak{V}$ have index set $J$, 
$\calG(\mathfrak{V})_0=\coprod_{j\in J}\tV_j$;   and similarly,
$\calG(\mathfrak{W})_0=\coprod_{k\in K}\tW_k$.

We will construct the space $M(\mathfrak{U})$ by constructing subspaces of the form
\[M(\mathfrak{U})_{jk}=\Gbmap^{-1}(\tV_j)\cap\Hbmap^{-1}(\tW_k),\]
together with the maps 
\[
\begin{tikzpicture}[xscale=1.8,yscale=-1.2]
    \node (A0_0) at (0, 0) {$\tV_j$};
    \node (A0_1) at (1, 0) {$M(\mathfrak{U})_{jk}$};
    \node (A0_2) at (2, 0) {$\tW_k$.};
    \path (A0_1) edge [->]node [auto] {$\scriptstyle{\Hbmap}$} (A0_2);
    \path (A0_1) edge [->]node [auto,swap] {$\scriptstyle{\Gbmap}$} (A0_0);
\end{tikzpicture}
\]

The space $M(\mathfrak{U})_{jk}$ is
constructed as a quotient of the space 
$\coprod_{i\in I}\mathsf{A}_{ji}^{\mathfrak{V}}\times\tU_i\times\mathsf{A}_{ki}^{\mathfrak{W}},$
where we give the modules $\mathsf{A}_{ji}^{\mathfrak{V}}$ and $\mathsf{A}_{ki}^{\mathfrak{W}}$
the discrete topology (they are empty whenever $U_i\nsubseteq V_j$, respectively $U_i\nsubseteq W_k$).
The equivalence relation $\sim$ on $\coprod\mathsf{A}_{ji}^{\mathfrak{V}}\times\tU_i\times
\mathsf{A}_{ki}^{\mathfrak{W}}$ is generated by 

\[(\lambda_{ji},\trho_{ii'}(\nu_{ii'})(x),\lambda_{ki})\sim
(\boldalpha_{jii'}(\lambda_{ji}\otimes\nu_{ii'}),x,\boldalpha_{kii'}(\lambda_{ki}\otimes\nu_{ii'}))\]
for any $\nu_{ii'}\in \ABST_{\mathfrak{U}}(\mu_{ii'})$.   (Note that this equivalence relation
is of the same form as the relation used to define the arrow spaces of the atlas groupoids.) Then,

\[M(\mathfrak{U})_{jk}:=\left(\coprod_{\substack{\tU_i\in\,\calU \\ U_i\subseteq V_j\cap W_k}}
\mathsf{A}_{ji}^{\mathfrak{V}}\times\tU_i\times\mathsf{A}_{ki}^{\mathfrak{W}}\right)/\sim.\]

The map $\Gbmap \colon M(\mathfrak{U})_{jk}\to \tV_j$ is defined by
$[\lambda_{ji},x,\lambda_{ki}]\mapsto\trho_{ji}(\lambda_{ji})(x)$
and the map $\Hbmap \colon M(\mathfrak{U})_{jk}\to \tW_k$ is defined by
$[\lambda_{ji},x,\lambda_{ik}]\mapsto\trho_{ki}(\lambda_{ki})(x)$
(with $\trho_{ji}$ and $\trho_{ki}$ as described in Remark~\ref{refinement-rmks}(\ref{D:rmk1})).  
Since $\mathfrak{U}$ is a refinement, these maps are surjective local diffeomorphisms and
in particular surjective submersions as required.

We define the manifold $M(\mathfrak{U})=\coprod_{j,k} M(\mathfrak{U})_{jk}$.

The right action of $\calG(\mathfrak{V})$ and the left action of $\calG(\mathfrak{W})$
are defined in a way analogous to composition in atlas groupoids.  
Let $g\in \calG(\mathfrak{V})_1$ with $s(g)\in\tV_{j'}$ and $t(g)\in\tV_j$, and let
$(\lambda_{ji},x,\lambda_{ki})$ represent an element of $M(\mathfrak{V})_{jk}$ with 
$\Gbmap([\lambda_{ji},x,\lambda_{ki}])=\trho_{ji}(\lambda_{ji})(x)=t(g)$.
Then $g\in\calG(\mathfrak{V})_1$ is represented by a triple $(\theta_{j'j''},y,\theta_{jj''})$
with $y\in\tV_{j''}$, $\theta_{j'j''}\in\ABST_{\mathfrak{V}}(\mu_{j'j''})$
and $\theta_{jj''}\in\ABST_{\mathfrak{V}}(\mu_{jj''})$.
Now $t(g)=\trho_{jj''}(\theta_{jj''})(y)$, so $\trho_{jj''}(\theta_{jj''})(y)=
\trho_{ji}(\lambda_{ji})(x)$.

By Remark~\ref{refinement-rmks}(\ref{strong_compatibility_2}), this implies that there are a
chart $\tU_{i'}$ in $\calU$, a point $z\in\tU_{i'}$, an abstract embedding
$\nu_{ii'}\in \ABST_{\mathfrak{U}}(\mu_{ii'})$ and an abstract embedding
$\lambda_{j''i'}\in\mathsf{A}^{\mathfrak{V}}_{j''i'}$ such that $\trho_{ii'}(\nu_{ii'})(z)=x$,
$\trho_{j''i'}(\lambda_{j''i'})(z)=y$ and $\boldalpha_{jj''i'}(\theta_{jj''}
\otimes\lambda_{j''i'})=\boldalpha_{jii'}(\lambda_{ji}\otimes \nu_{ii'})$.
Hence, $(\lambda_{ji},x,\lambda_{ki})\sim
(\boldalpha_{jii'}(\lambda_{ji}\otimes\nu_{ii'}),z,\boldalpha_{kii'}(\lambda_{ki}\otimes\nu_{ii'}))$.

Then the right action of $g$ on the class of the point $(\lambda_{ji},x,\lambda_{ki})$ is represented
by $(\lambda_{ki},x,\lambda_{ji})\cdot g:= (\boldalpha_{j'j''i'}(\theta_{j'j''}\otimes\lambda_{j''i'}),
z,\boldalpha_{kii'}(\lambda_{ki}\otimes\nu_{ii'}))$.  
It can be verified that this is independent of the choice of representatives
$(\lambda_{ji},x,\lambda_{ki})$ in $M(\mathfrak{U})_{jk}$ and $(\theta_{j'j''},y,\theta_{jj''})$
in $\calG(\mathfrak{V})_1$.

The left action by $\calG(\mathfrak{W})$ on $M(\mathfrak{U})$ is defined
in a similar (but dual) fashion.

It is a straightforward  calculation to check that this satisfies all the conditions to be a
Morita equivalence.
\end{proof}

So we conclude that the notion of orbifold defined in terms of orbifold atlases and atlas
equivalences as presented in this paper corresponds to the notion of orbifold defined in terms
of proper \'etale groupoids and Morita equivalence.


\end{document}